\documentclass{amsart}
 \RequirePackage[colorlinks,citecolor=blue,urlcolor=blue,linkcolor=blue]{hyperref}
\makeatletter
\renewcommand*\env@matrix[1][\arraystretch]{%
  \edef\arraystretch{#1}%
  \hskip -\arraycolsep
  \let\@ifnextchar\new@ifnextchar
  \array{*\c@MaxMatrixCols c}}
\makeatother

\usepackage{stmaryrd}
\newcommand{\Zin}[1]{\llbracket #1 \rrbracket}

\usepackage[top=1 in,bottom=1 in, left=1 in, right= 1 in, includehead,includefoot]{geometry}
\usepackage{enumerate}

 \usepackage{framed,color,pgf}

 \usepackage{framed,pgf}
\newcounter{oldeq}
\newcounter{usesofarxiv}
 \newcommand{\arxiv}[1]{
\setcounter{oldeq}{\value{equation}}
 \addtocounter{usesofarxiv}{1}
 \setcounter{equation}{0}
\def\theoldeq{\theequation}
\def\theequation{x-\arabic{usesofarxiv}.\arabic{equation}}
\def\theequation{\arabic{section}.\arabic{usesofarxiv}.\arabic{equation}}
\def\theequation{\thesection.\arabic{usesofarxiv}.\arabic{equation}}
  \colorlet{shadecolor}{gray!10}
{
\begin{shaded}
\footnotesize
#1  \normalsize
\end{shaded}
   \setcounter{equation}{\value{oldeq}}
\numberwithin{equation}{section}
}}

\newcommand{\ffloor}[1]{\lfloor\!\lfloor{#1}\rfloor\!\rfloor}

 \usepackage{tikz}
 \usetikzlibrary{patterns}
\usepackage{natbib}
\usepackage{amsmath,amsthm,amsfonts,amssymb,fancybox,float}
 \usepackage{framed}
 \usepackage{color,soul}
 \usepackage{url}
\newcommand{\comment}[1]{}
\newcommand{\commentYZ}[1]{}
\newcommand{\longcomment}[1]{}
\renewcommand{\longcomment}[1]{\ovalbox{\begin{minipage}{.9\textwidth}\color{blue}#1\end{minipage}}}
\renewcommand{\comment}[1]{{\color{magenta}\ovalbox{#1}}}
\newcommand{\hide}[1]{{\color{blue}\ovalbox{\tiny (hidden text)}}}

\def\fddto{\stackrel{\rm f.d.d.}{\Longrightarrow}}
\newcommand{\ind}{{\bf 1}}

\def\g{{\mathsf q}}

\newcommand{\eqnh}{\begin{eqnarray*}}
\newcommand{\eqne}{\end{eqnarray*}}
\newcommand{\eqnhn}{\begin{eqnarray}}
\newcommand{\eqnen}{\end{eqnarray}}
\newcommand{\equh}{\begin{equation}}
\newcommand{\eque}{\end{equation}}

\def\topp#1{^{(#1)}}

\def\ccbb#1{\left\{#1\right\}}

\def\pp#1{\left(#1\right)}

\def\bb#1{\left[#1\right]}

\def\floor#1{\left\lfloor #1 \right\rfloor}

\def\ceil#1{\left\lceil #1 \right\rceil}

\def\vv#1{{\boldsymbol #1}}

\def\wt#1{\widetilde{#1}}

\def\limn{\lim_{n\to\infty}}

\def\weakto{\Rightarrow}
\def\toD{\weakto}
\def\Z{{\mathbb Z}}

\def\R{{\mathbb R}}

\def\calD{\mathcal D}

\def\fddto{\xrightarrow{\textit{f.d.d.}}}
\renewcommand{\ind}{{\bf 1}}

\newcommand{\hidecomment}[1]{}
\newtheorem{theorem}{Theorem}[section]
\newtheorem{corollary}[theorem]{Corollary}
\newtheorem{lemma}[theorem]{Lemma}
\newtheorem{proposition}[theorem]{Proposition}

\theoremstyle{definition}

\theoremstyle{remark}
\newtheorem{remark}[theorem]{Remark}
\theoremstyle{remark}

\newcommand{\calP}{{\mathcal P}}

\newcommand{\TT}{\mathbb{T}}

\def\<{\langle}
\def\>{\rangle}

\usepackage{dsfont}
\newcommand{\RR}{\mathds{R}}

\newcommand{\ZZ}{\mathds{Z}}

\newcommand{\eps}{\varepsilon}

\newcommand{\Var}{\mathds{V}\!\mathrm{ar}}

\usepackage{graphicx,fancybox,pgf,amsmath,float,amssymb}
\usepackage{fancybox,graphicx,dsfont,pgf}
\usepackage{subfigure,color,wrapfig}

\usepackage{latexsym}

\newcommand{\EE}{\mathds{E}}
\renewcommand{\Pr}{\mathds{P}}

\def\<{\langle}
\def\>{\rangle}

\numberwithin{equation}{section}

\date{Created   Monday, February 13, 2023. Current file \jobname.tex. Printed: \today}
\title{Pitman's discrete $2M-X$ theorem for arbitrary initial laws\\ and   continuous time limits}
\author{W\l odzimierz Bryc}
\address
{
W{\l}odzimierz Bryc\\
Department of Mathematical Sciences\\
University of Cincinnati\\
2815 Commons Way\\
Cincinnati, OH, 45221-0025, USA.
}
\email{wlodzimierz.bryc@uc.edu}

\author{Jacek Weso{\l}owski}
\address{Jacek Weso{\l}owski, Faculty of Mathematics and Information Science,
Warsaw University of Technology, pl. Politechniki 1, 00-661
Warszawa, Poland}
\email{jacek.wesolowski@pw.edu.pl}

\numberwithin{equation}{section}

\begin{document}\sloppy

\begin{abstract}
We discuss Pitman’s representation of a Markov process, which serves as a discrete analog to the Bessel 3D process starting at time 0 from an arbitrary initial law. This representation involves maxima of lazy simple random walks  and an auxiliary independent random variable.
 The law of the auxiliary random variable is explicitly related to the initial law of the Markov process.
 The proof is kept at an elementary level and relies on a reconstruction formula for the generalized Pitman transform.

We then use continuous-time limits to shed additional light on the relation between two representations of the Bessel 3D process that appeared in the description of the stationary measure of the KPZ fixed point on the half-line, as proposed in %
\cite{barraquand2022steady}.

\end{abstract}
\maketitle

\section{Introduction}
Let $\pp{S_t}_{t=0,1,\dots}$ be a symmetric random walk with increments $\pm 1$ and $M_t=\max_{j\in\{0,1,\dots,t\}} \,S_j$.
 Pitman \cite[Lemma 3.1]{pitman1975one} proved that  $\pp{Y_t}_{t=0,1,\dots}:=\pp{2M_t-S_t}_{t=0,1,\dots}$ is a discrete-time Markov process  started at $X_0=0$ with transition probabilities
 \begin{equation}
   \label{X-trans}
    \Pr(Y_{t+1}-Y_t=\delta |Y_t=k)=\tfrac12 \begin{cases}
   \frac{k+2}{k+1} & \delta =1, \\
   \frac{k}{k+1} & \delta =-1,
 \end{cases}\qquad t=0,1,\dots.
 \end{equation}
  As part of  the proof, he also established that the conditional law of $S_n$ given $(Y_0,Y_1,\dots,Y_n)$ is uniform on the set $\{-Y_n,-Y_n+2,\dots,Y_n-2,Y_n\}$, see \cite[Theorem 1.1]{Chapon-Chhaibi-2021} for an explicit statement.

From \eqref{X-trans} it follows that the Pitman transform $\pp{S_t}_{t=0,1,\dots}\mapsto \pp{2M_t-S_t}_{t=0,1,\dots}$ maps a random walk to a Markov process with values in $\ZZ_{\geq 0}=\{0,1,\ldots\}$, which can be interpreted as a random walk conditioned on remaining nonnegative. Random walks conditioned on positivity have been studied in numerous references, including \cite{keener1992limit}, \cite{bertoin1994conditioning}, resulting in processes with transition probabilities \eqref{X-trans} and \eqref{Z-trans} (with $\sigma=0$).
 Ref. \cite{OConnell2002representation}  pointed out that
  their results %
yield the following discrete analogue of Pitman’s lemma: if $\pp{S_t}_{t= 0,1,\ldots}$ is a simple random walk with
non-negative drift (in continuous or discrete time) and $M_t = \max_{ j\in\{0,1,\ldots, t\}} S_j$, then
$\pp{2M_t-S_t}_{t=0,1,\ldots}$ has
the same law as that of $\pp{S_t}_{t=0,1,\ldots}$ conditioned to stay positive (in the case of a symmetric random walk, this conditioning is in the sense of Doob).
  \cite{hambly2001pitman} relax independence assumption, proving that the process $\pp{2M_t-S_t}_{t=0,1,\ldots}$ has the same law as that of $\pp{S_t}_{t=0,1,\ldots}$
conditioned to stay non-negative when $\pp{S_t-S_{t-1}}_{t=0,1,\ldots}$ is a $\pm1$-valued Markov process.
The Pitman transform
appears also in the description of the dynamics of the box-ball system in
\cite{Croydon2023}.

 This transform inspired numerous   modifications   and generalizations, see \cite[Section 3.4]{corwin2014tropical}  for a concise summary, including relations to the Robinson--Schensted--Knuth correspondence studied in  \cite{o2003conditioned,Oconnell-2003-trans}.
  \cite{miyazaki1989theorem} and \cite{tanaka1989time}  generalized Pitman's result, describing Markov processes that arise by the same representation from nonsymmetric simple random walks. In particular, \cite[Proposition 2]{miyazaki1989theorem} derived transition probabilities \eqref{Z-trans} when $\sigma=0$.  \cite{mishchenko2006discrete} established  additional properties of $\pp{Y_t}_{t=0,1,\dots}$,  including a discrete version of a theorem from \cite{Imhof1992}.
 Biane,
Bougerol and O’Connell in \cite{biane2005littelmann} and \cite{biane2009continuous} have extended Pitman’s theorem in the context of semisimple
Lie algebras and finite Coxeter groups. More recently,
\cite{bougerol2018pitman} and \cite{defosseux2023converse} have established Pitman-type
theorems in the context of affine Lie algebra. \cite{corwin2024periodic}  used periodic Pitman transforms to study invariant measures for KPZ. There is also a
 geometric version of the Pitman transform that was introduced in \cite{matsumoto1999some} and
was studied in \cite{donati2001some} and \cite{hariya2004limiting}. Its relation to the geometric RSK correspondence is discussed in \cite[Section 3.4]{corwin2014tropical}.

Pitman's lemma yields the following folklore result,  which we were not able to locate explicitly in the literature.
  \begin{theorem}\label{Thm-reviewer}
    If $(S_t)_{ t= 0,1,\dots}$ is a   symmetric  random walk on $\ZZ$ starting from $0$ and $G$ is an independent random variable, which is   uniformly distributed on the set $\{0,1,\dots, g\}$  for some non-negative integer $g$, then
    \begin{equation}\label{Pit-G}
      X_t:= g+ 2 (M_t-G)_+-S_t ,\quad t=0,1,\dots,
    \end{equation}
    is a discrete Bessel process  (that is, its transition probabilities are given by \eqref{X-trans}) starting from $X_0=g$.
  \end{theorem}

 Theorem \ref{Thm-reviewer} can be derived from Pitman's lemma  by stopping $2M-S$ process at the first time it reaches integer $g>0$, compare
 \cite[Theorem 2]{mishchenko2006discrete}.
    (Theorem \ref{Thm-reviewer} follows also from more general Theorem \ref{Thm1}, which we prove independently of Pitman's lemma.)
  \arxiv{  \begin{proof}
Fix $g\in\ZZ_{\geq 0}$. Take $t_*$ large enough (and of the right parity) so that $\Pr(S_{t_*}=g)>0$.    Therefore, with the Pitman transform $(Y_t=2M_t-S_t)_{t=0,1,\ldots}$   as above,  $\Pr(Y_{t_*}=g)>0$ and the conditional law of $S_{t_*}$ given $A_g:=\{Y_0=0, Y_1=y_1 ,\dots,Y_{t_*}=g\}$ of positive probability, is uniform on the set $\{-g, -g+2, \dots g-2, g\}$.

Let  $X_t=Y_{t+t_*}$ and  %
 $G=\max_{k\leq t^*} S_k-S_{t_*}$. Then $G=\tfrac12 (Y_{t_*} -S_{t_*})$ is uniform on $\{0,1,\dots, g\}$ conditionally on $A_g$.
We will verify that \eqref{Pit-G} holds conditionally on $A_g$.

 Since %
 $\max \{a,b\}=a+(b-a)_+$ we have

\begin{multline}\label{Review2}
  Y_{t+t_*}= 2 \max_{k\leq t+t_*} S_k - S_{t+t_{*}} = 2 \max_{k\leq t_*} S_k   +
  2 (\max_{j\leq t}S_{j+t_*}-\max_{k\leq t_*} S_k)_+  - S_{t+t_{*}}
 \\ =
   Y_{t_*}  +
  2 (\max_{j\leq t}(S_{j+t_*}-S_{t_*})-G)_+  - ( S_{t+t_{*}} -S_{t_*})
  =Y_{t_*}  +
  2 (\max_{j\leq t}(S_{j}^*)-G)_+  -  S_{t}^*
\end{multline}
where $S_{t}^*=S_{t+t_*}-S_{t_*}$.   Note that $(S_t^*)_{t=0,1,\ldots}$ and  $(G,A_g)$ are independent, and $(S_t^*)_{t=0,1,\ldots}\stackrel{d}{=}(S_t)_{t=0,1,\ldots}$.  Therefore,
\eqref{Review2} says that
\begin{equation}
    \label{Reviewer3}
    X_t=X_0+2 (\max_{j\leq t} S_j^*-G)_+-S_t^*
\end{equation}
conditionally on $A_g$ has the same law as the right-hand side of \eqref{Pit-G}, and in particular $\Pr(X_0=g|A_g)=1$.

On the other hand, transition probabilities for $(X_t)_{t=0,1,\ldots}$ as defined by the right-hand side of \eqref{Reviewer3} do not depend on conditioning on $A_g$ and are the same as transition probabilities for $(Y_{t_*+t})_{t=0,1,\ldots}$, which are given by \eqref{X-trans}.
  \end{proof}

   }

Related results in   continuous-time appeared recently in the description of the stationary measures for the hypothetical KPZ fixed point on the half-line in Refs. \cite{barraquand2022steady,Bryc-Kuznetsov-2021}.  In discrete-time, Markov processes with transition probabilities given by \eqref{Z-trans} and  special   initial laws arise also as limits of random Motzkin paths in \cite{Bryc-Wang-2023b}.

Our goal is to clarify these relations by providing explicit formulas that extend these examples to more general initial laws.  For the proofs in continuous-time, we rely on discrete approximations. For the discrete-time results, we use elementary and explicit formulas for the finite-dimensional distributions.

\color{black}

For $\rho>0$ and $\sigma\geq 0$ we consider the following three processes: a random walk $(S_t)_{t=0,1,\dots}$, starting at $S_0=0$, and with $S_t=\sum_{j\leq t}\xi_j$, where $\xi_j$ are i.i.d. random variables with the law %
\begin{equation}\label{xi-gen}
  \Pr(\xi_1=\delta )=\tfrac{1}{\rho+\sigma+1/\rho}\begin{cases}
    1/\rho & \delta =1, \\
    \sigma & \delta =0, \\
    \rho & \delta =-1,
  \end{cases}
\end{equation}
  a random walk $(\wt S_t)_{t=0,1,\dots}$, starting at  $\wt S_0=0$, and with
 $\wt S_t=\sum_{j=1}^t \wt \xi_j$ where $\wt \xi_j$ are i.i.d. random variables with the law  of $-\xi_1$
and a Markov process $(X_t)_{t=0,1,\dots}$ with transition probabilities %
    \begin{equation}\label{Z-trans}
  \Pr(X_{t+1}=k+\delta |X_t=k)=\Pr(\xi_1=\delta )\,\frac{[k+\delta +1]_{\rho^2}}{[k+1]_{\rho^2}}, %
  \end{equation}
and with arbitrary initial law supported in $\ZZ_{\geq 0}$. Here $[n]_q=1+\dots+q^{n-1}$.
We will occasionally extend this notation to $a\in\RR$,
$$[a]_q=\begin{cases}
    \frac{1-q^a}{1-q} & q \ne 1 \\
    a & q=1.
\end{cases}.
$$
  We note that  $(S_t)_{t=0,1,\dots} \stackrel{\mathcal D}{=}(-\wt S_t)_{t=0,1,\dots}$, and that if $\rho=1$, then  $(S_t)_{t=0,1,\dots}$ is a symmetric (lazy) random walk,
  and transition probabilities \eqref{Z-trans} become
  \begin{equation*}%
     \Pr\pp{X_{t+1}=k+\delta \middle|X_t=k}=\tfrac1{2+\sigma} \begin{cases}
   \frac{k+2}{k+1} & \delta =1, \\
  \sigma & \delta =0, \\
   \frac{k}{k+1} & \delta =-1,
 \end{cases}\qquad t=0,1,\ldots,
  \end{equation*} in agreement with \eqref{X-trans} when $\sigma=0$.
  We also note that   formula \eqref{xi-gen}  encompasses  all the laws on $\{0,\pm 1\}$ that assign a positive mass to $+1$ and $-1$.

We consider arbitrary initial laws for $X_0 \in\ZZ_{\geq 0}$.   This law  affects the $2M-S$ representation   of $(X_t-X_0)$ through an
auxiliary independent random variable  $G$ with an appropriate law related to that of $X_0$.
The following extends Theorem \ref{Thm-reviewer}.

\begin{theorem}\label{Thm1}
Assume $\rho>0$ %
  and $\sigma\geq 0$.
Let $G$ be a random variable with values in $\ZZ_{\geq 0}$ such that $G$ and $(S_t)_{t=0,1,\dots}$ are independent.
 If $(X_t)_{t=0,1,\dots}$ is a Markov process with transition probabilities \eqref{Z-trans},  then the following conditions are equivalent:
  \begin{enumerate}
    [(i)]
    \item
    \begin{equation}
      \label{Pit-rep}
       (X_t-X_0)_{t=0,1,\dots} \stackrel{\calD}{=}(2(M_t-G)_+-S_t)_{t= 0,1,\dots},
    \end{equation}
     where   $ M_t=\max_{j\in\{0,1,\ldots, t\}} S_j$.
    \item
     The law of $G$ is
\begin{equation}
    \label{q-damage}
         \Pr(G=n)=\rho^{2n}\sum_{j=n}^\infty \tfrac{\Pr(X_0=j)}{[j+1]_{\rho^2}} ,\quad n=0,1,\dots.
\end{equation}

  \end{enumerate}

\end{theorem}

Implication (ii)$\Rightarrow$(i) can be obtained from Theorem \ref{Thm-reviewer} and discrete Girsanov transformation \eqref{Girsanov}. In our more direct approach we compare the finite dimensional distributions of the processes in \eqref{Pit-rep}.

Since random walk $\wt S$ is the same as random walk $S$ with $1/\rho$ instead of $\rho$ and transition probabilities \eqref{Z-trans}  are given by the same expression in both cases, we have the following equivalent version of Theorem \ref{Thm1}.
\begin{remark}\label{Rem:Thm1}
    Let $\wt G$ be a random variable with values in $\ZZ_{\geq 0}$ such that $\wt G$ and $(\wt S_t)_{t=0,1,\dots}$ are independent. Then the following conditions are equivalent:
  \begin{enumerate}
    [(i)]
    \item
    \begin{equation}
      \label{Pit-rep2}
       (X_t-X_0)_{t=0,1,\dots} \stackrel{\calD}{=}(2(\wt M_t-\wt G)_+-\wt S_t)_{t= 0,1,\dots},
    \end{equation}
    where    $\wt M_t=\max_{j\in\{0,1,\ldots, t\}}\wt S_j$.
     \item
     The law of $\wt G$ is
 \begin{equation}
     \label{wt-G-dam}
      \Pr(\wt G=n)=\rho^{-2n}\sum_{j=n}^\infty \tfrac{\Pr(X_0=j)}{[j+1]_{1/\rho^2}} ,\quad n=0,1,\dots.
 \end{equation}

  \end{enumerate}
\end{remark}
\color{black}
We remark that with $X_0$, $G$ and $\wt G$ defined on a common probability space the relation \eqref{q-damage} can be interpreted as specifying the conditional law of $G$ given $X_0$ as

 \begin{equation}\label{q-thin}
      \Pr( G=m|X_0=n)=\frac{\rho^{2m}}{[n+1]_{\rho^2}} %
      \quad m=0,1,\dots,n,\; n=0,1,\dots
    \end{equation}
Similarly,   \eqref{wt-G-dam} is equivalent to specification of the conditional law of random variable $\wt G$ given $X_0=n$ by the right hand side of \eqref{q-thin}  with $\rho$ changed to $1/\rho$.  Moreover, calculation shows that $\wt G\stackrel{\mathcal D}{=}X_0-G$.
It is clear that  the original Pitman  result corresponds to the case $\rho=1$, $X_0=0$, $G=0$, and $\sigma=0$.  (For an extension to non-symmetric random walk, see \cite{miyazaki1989theorem}.)

We remark that if  \eqref{q-damage} and \eqref{wt-G-dam} hold   then Theorem  \ref{Thm1} and Remark \ref{Rem:Thm1} trivially imply
\begin{equation}\label{H-Y-D}
    (2(M_t-G)_+-S_t)_{t= 0,1,\dots}\stackrel{\calD}{=}(2(\wt M_t-\wt G)_+-\wt S_t)_{t= 0,1,\dots},
\end{equation}
where  $G,(S_t)_{t=0,1,\dots}$ are independent and $\wt G,(\wt S_t)_{t=0,1,\dots}$ are independent.

It is also of interest to note that there are other processes beyond $\pp{X_t}_{t=0,1,\ldots}$ that can be represented by the transformation $2(M-G)_+-S$.
An example of such a process is the random walk $\pp{ S_t}_{t=0,1,\ldots}$ itself, see Remark \ref{Rem-S}.

 Next, we give an interpretation of the process $(X_t-X_0)_{t=0,1,\ldots}$ as a random walk conditioned on staying above an independent random level.

\begin{theorem}\label{Thm1-2} Assume $\sigma\geq 0$.
Let $V$ be a random variable with values in $\ZZ_{\geq 0}$ such that $V$ and $(S_t)_{t=0,1,\dots}$ are independent.  If  $0<\rho<1$,
 and $(X_t)_{t=0,1,\dots}$ is a Markov process with transition probabilities \eqref{Z-trans},  then the following conditions are equivalent:
\begin{enumerate}[(i)]
    \item The law of $(X_t-X_0)_{t=0,1,\dots}$ is the same as the law of $(S_t)_{t=0,1,\dots}$ conditioned on
     $(S_t+V)_{t=0,1,\dots}$  to remain non-negative.

    \item The law of $V$ is \[\Pr(V=k)\propto \tfrac{\Pr(X_0=k)}{[k+1]_{\rho^2}},\quad k=0,1,\dots.\]
\end{enumerate}

\end{theorem}

Instead of deriving implication (ii)$\Rightarrow(i)$ from the results on random walks conditioned on remaining non-negative, we give a direct elementary proof.

Noting that if  $\rho>1$, then  $1/\rho<1$,  we get the following version of Theorem \ref{Thm1-2}.
\begin{remark}\label{Rem:Thm1-2}
Let $\wt V$ be a random variable with values in $\ZZ_{\geq 0}$ such that $\wt V$ and $(\wt S_t)_{t=0,1,\dots}$ are independent.    If  $\rho>1$, then the following conditions are equivalent:
\begin{enumerate}[(i)]
    \item The law of $(X_t-X_0)_{t=0,1,\dots}$ is the same as the law of $(\wt S_t)_{t=0,1,\dots}$ conditioned on
     $(\wt S_t+\wt V)_{t=0,1,\dots}$  to remain non-negative.

    \item The law of $\wt V$ is
    \[\Pr(\wt V=k)\propto \tfrac{\Pr(X_0=k)}{[k+1]_{1/\rho^2}},\quad k=0,1,\dots.\]
\end{enumerate}
\end{remark}
\color{black}

For a given law of $X_0$ and $\rho\in(0,1)$, the laws of random variables $G$ in Theorem \ref{Thm1} and $V$ in Theorem \ref{Thm1-2} are related by
$$ C \Pr(G=n)= \rho^{2n} \Pr(V\geq n), \quad n=0,1,\dots$$
 where   $C$ is the proportionality constant from Theorem \ref{Thm1-2}.
In particular,    $\Pr(V=k)=C {\Pr(X_0=k)}/{[k+1]_{\rho^2}}$ with $C=1/\Pr(G=0)$.

We now point out a special case of interest, a process $(Z_t)_{t=0,1,\dots}$  that appears as the limit of random Motzkin paths in \cite[Corollary 1.9]{Bryc-Wang-2023b}.
 We write $G\sim \mathrm{geo}(p)$ if $G$ is a geometric random variable with parameter $p\in[0,1)$, that is, $\Pr(G=n)=(1-p)p^n$, $n=0,1,\dots$. The transition probabilities of $(Z_t)_{t=0,1,\ldots}$  are as in \eqref{Z-trans} and the initial law is
\begin{equation}\label{Ini-rho}
  Z_0 \stackrel{\mathcal D}{ =} G+\wt G,
\end{equation}
where $G$ and $\wt G$ are independent $G\sim \mathrm{geo}(\rho_0\rho)$ and $\wt G\sim \mathrm{geo}(\rho_0/\rho)$ with  $1<\rho<1/\rho_0$.
(Note that this inequality implies $\rho_0<1$, so we get $\rho_0<\rho<1/\rho_0 $ and hence $\rho\rho_0<1$ and $\rho_0/\rho<1$.) That is,   $Z_0$  has the  $q$-negative binomial law,
see  \cite[Theorem 3.1 for  $k=2$]{charalambides2016discrete}, i.e.,
\begin{equation}
    \label{dunkl}\Pr(Z_0=m)= [m+1]_q \theta^m(1-\theta)(1-\theta q), \; m=0,1,\dots
\end{equation}
with $q=\rho^2$ and $\theta=\rho_0/\rho$.

\begin{corollary}\label{C-BD-T2.8}
Let $1<\rho<1/\rho_0$. Then the law of
$ (Z_t-Z_0)_{t=0,1,\dots}$ is the same as
\begin{enumerate}[(i)]
\item  the law of $(2(M_t-G)_+-S_t)_{t= 0,1,\dots}$, where    $(S_t)_{t=0,1,\dots}$ and $G\sim \mathrm{geo}(\rho_0\rho)$ are independent.
\item   the law of $(\wt S_t)_{t=0,1,\dots}$  conditioned on
     $(\wt S_t+\wt V)_{t=0,1,\dots}$  to  remain non-negative,
   where   $(\wt S_t)_{t=0,1,\dots}$ and $\wt V \sim \mathrm{geo}(\rho_0\rho)$ are independent.
\end{enumerate}

\end{corollary}
 \begin{proof}
  In view of \eqref{Ini-rho}, to prove the first part we use \eqref{dunkl} to verify \eqref{q-thin}.   For $n,m\in\mathds Z_{\ge 0}$,  $n\le m$ we have
        \begin{equation*}
        \Pr\pp{G=n\middle|G+\wt G=m}=\tfrac{\Pr(G=n)\cdot\Pr(\wt G=m-n)}{\Pr(G+\wt G=m)}
     =\tfrac{(1-\rho_0\rho) (\rho_0\rho)^n\cdot(1-\rho_0/\rho)(\rho_0/\rho)^{m-n}}{
      (1-\rho_0\rho)(1-\rho_0/\rho)(\rho_0/\rho)^m[m+1]_{\rho^2}
      }
      = \tfrac{\rho^{2n}}{[m+1]_{\rho^2}}.
    \end{equation*}
    To prove the second part we apply Theorem \ref{Thm1-2}(ii)   and rely on \eqref{dunkl} to get
    \[\Pr(\wt V =m)=c \frac{\Pr(Z_0=m)}{[m+1]_{1/\rho^2}}=c\frac{(1-\rho_0\rho)(1-\rho_0/\rho)(\rho_0/\rho)^m[m+1]_{\rho^2}} {[m+1]_{\rho^2}/\rho^{2m+2}} \propto (\rho_0\rho)^m.\]
  \end{proof}

\begin{remark} \label{R-1/rho}
In particular, if $0\leq \rho_0\rho<1$ and $\rho_0/\rho<1$
then \eqref{H-Y-D} holds.

\end{remark}

    \begin{remark}\label{C-BD-T2.6}
Suppose $\rho=1$.  Then \eqref{Pit-rep} and \eqref{Pit-rep2} hold when $G\sim \mathrm{geo}(\rho_0)$, $\wt G\sim \mathrm{geo}(\rho_0)$ and then $X_0\stackrel{\mathcal D}{=}G+\wt G$, with $G$ and $\wt G$ independent, is
 a negative binomial random variable,
\begin{equation}\label{NB2}
  \Pr(X_0=n)=(1-\rho_0)^2(n+1) \rho_0^n,\; n=0,1,\dots
\end{equation}
  Furthermore,
\[
\Pr(G=m|G+\wt G=n)=\Pr(\wt G=m|G+\wt G=n)=\tfrac{1}{n+1},\quad m=0,1,\ldots,n; \; n=0,1,\ldots.
\]

  \end{remark}

With the initial law \eqref{Ini-rho},  process $(X_t)_{t=0,1,\dots}$ is a discrete analog of the continuous-time process in \cite[Theorem 2.8]{Bryc-Kuznetsov-2021}, and process  $(X_t-X_0)_{t=0,1,\dots}$  is an analog of the process appearing in  \cite{barraquand2022steady} as described in   \cite[Section 2.1 iii(b) and Remark 2.9]{Bryc-Kuznetsov-2021}.
With the initial law \eqref{NB2},  process $(X_t)_{t=0,1,\dots}$ is a discrete analog of the continuous-time process in \cite[Theorem 2.6]{Bryc-Kuznetsov-2021} and process  $(X_t-X_0)_{t=0,1,\dots}$  is an analog of the process appearing in  \cite{barraquand2022steady} as described in   \cite[Section 2.1 iii(a) and Remark 2.7]{Bryc-Kuznetsov-2021}.

 Theorem \ref{Thm1} can be applied together with Donsker's theorem to establish a continuous-time limit. %

\begin{theorem}
  \label{T-YZ} \sloppy
  For $N=1,2,\dots$, let $\rho=1-v/\sqrt{N}+o(1/\sqrt{N})$, where $v\in\RR$ is fixed.
  Consider a sequence $(X_t\topp N)_{t=0,1,\dots}$ of Markov processes with transition probabilities \eqref{Z-trans} and with initial laws such that
  $$\frac{X_0\topp N}{\sqrt{N}}\toD \mu \mbox{ as $N\to\infty$}
  $$
  for some probability measure $\mu$.
  Then we have the following weak convergence in Skorokhod's space $D[0,\infty)$:
\begin{equation}\label{Weak-lim+}
\frac{1}{\sqrt{N}}\left(X_{\floor {N t}}\topp N-X_0\topp N\right)_{t\geq 0}\toD \left(2 \pp{\sup_{0\leq s\leq t} B_s\topp v-\gamma}_+-B_t\topp v\right)_{t\geq 0} \mbox{ as $N\to\infty$},
\end{equation}
where   $\left(B_t\topp v\right)_{t\ge 0}$ and $\gamma$ are independent, $B_t\topp v=W_{2t/(2+\sigma)}\topp v$ with  $W_t\topp v=W_t+vt$ denoting the Wiener process with drift $v$ and  random variable $\gamma\geq 0 $ has the law
   with %
   the density %
\begin{equation}\label{density}
    f_\mu \topp v(x):=
   e^{-2vx}\int_{[x,\infty)}\,g_v(y)\,\mu(dy)\quad \mbox{on $(0,\infty)$,  where} \;g_v(y):=\begin{cases} \tfrac{2v}{1-e^{-2vy}} & v\ne 0\\
   \tfrac{1}{y} & v=0 \end{cases}
\end{equation}
 and possibly with an atom at $0$ of mass
$1-\int_0^\infty f_\mu\topp v(x)dx\geq 0$.
\end{theorem}
 (In particular, if  $\mu=\delta_0$, then $f_\mu\topp v\equiv 0$   and $\gamma= 0$ a.~s.)
The proof of Theorem \ref{T-YZ} is in Section \ref{T-YZ-proof}.

The following result is a direct consequence of  Theorem \ref{T-YZ} and Corollary \ref{C-BD-T2.8}.

\begin{corollary}\label{T-lim}
Let $u>0$, $v\leq 0$ be such that $u+v>0$. For $N=1,2,\dots$, let $\rho_0=1-u/\sqrt{N}+o(1/\sqrt{N})$, $\rho=1-v/\sqrt{N}+o(1/\sqrt{N})$.
Consider a sequence of Markov processes  $(Z_t\topp N)_{t=0,1,\dots}$ with  transition probabilities \eqref{Z-trans} and with initial law \eqref{Ini-rho}.
Then we have the following weak convergence in Skorokhod's space $D[0,\infty)$:
\begin{equation}\label{Weak-lim}
\frac{1}{\sqrt{N}}\pp{Z_{\floor {N t}}\topp N-Z_0\topp N}_{t\geq 0}\toD \pp{2 \pp{\sup_{0\leq s\leq t} B_s\topp v-\tfrac{1}{u+v}\gamma}_+-B_t\topp v}_{t\geq 0} \mbox{ as $N\to\infty$},
\end{equation}
    where $\gamma$ is an exponential random variable with parameter 1, $\pp{B_t\topp v}_{t\ge 0}$ and $\gamma$ are independent, $B_t\topp v=W_{2t/(2+\sigma)}\topp v$ with  $W_t\topp v=W_t+vt$ denoting the Wiener process with drift $v$.
\end{corollary}
\begin{proof} If $G\topp N$ is a sequence  of random variables with law  $\mathrm{geo}(\rho_0\rho)$ such that $\rho_0\rho_1=1-(u+v)/\sqrt{N}+o(1/N)$ then $G\topp N/\sqrt{N}\toD \tfrac{1}{u+v}\gamma$ where $\gamma$ is exponential with density $e^{-x}\ind_{x>0}$.
Indeed, $$\Pr\pp{\tfrac{1}{\sqrt{N}}G\topp N\geq x}=\pp{1-\tfrac{1}{\sqrt{N}}(u+v)+o(\tfrac1N)}^{\floor{x \sqrt{N}}}\to e^{-(u+v)x}.$$

The result now follows from formula \eqref{Pit-rep} which holds with $G=G\topp N$ by Corollary \ref{C-BD-T2.8}.
\end{proof}

  By Remark \ref{R-1/rho},  the law of the limiting process in \eqref{Weak-lim} is invariant with respect to the change $v\mapsto -v$.
Since $W$ and $-W$ have the same law,  we see that
this limit for $\sigma=2$ appeared as the non-Gaussian part of the (hypothetical) stationary measure of the KPZ fixed point on the half-line in \cite[(36), (37)]{barraquand2022steady}.
    We also remark that the latter of these formulas
    with their parameters $\wt u=-\wt v$  refers to  the
    process
    \begin{equation*}
        \left(2 \pp{\sup_{0\leq s\leq t}W_{s/2}\topp u-\tfrac{1}{2u}\gamma}_+-W_{t/2}\topp u \right)_{t\geq 0}, \quad u>0,
    \end{equation*}
    which, as they noted, %
    has the same law as  $\pp{W_{t/2}\topp{u}}_{t\geq 0}$, compare \cite[Theorem, 2.2]{williams1974path}. See Remark \ref{Rem-S} for a discrete version of this fact.

Next we present a continuous-time version of  Theorem \ref{Thm1}.
Consider the continuous-time Markov process $\pp{Z_t}_{t\geq 0}$ with transition probabilities
\begin{equation}
    \label{ZZ-trans}
\Pr(Z_t=dy|Z_s=x)%
= \begin{cases}
    \tfrac{ \sinh(v y)}{ \sinh(v x)}\;e^{-v^2 (t-s)/2}\,\g_{t-s}(x,y) dy & v\ne 0, \\
    \\
    \tfrac{ y}{ x}\; \g_{t-s}(x,y) dy & v=0,
\end{cases}
\end{equation}
 where $ x>0, y>0$ and  $\g_t(x,y)$ is given by
 \begin{equation}
  \label{eta:heat}\g_t(x,y)=\tfrac{2}{\pi}\int_0^\infty e^{-t u^2/2}\sin(x u)\sin(y u)d u = \tfrac{1}{\sqrt{2\pi t}} \Big(e^{-\frac{(x-y)^2}{ 2t}}-e^{-\frac{(x+y)^2}{ 2t}} \Big),\; x,y\geq 0,\; t>0.
\end{equation}

(This is a Bessel 3D  process which appears in  \cite{Bryc-Kuznetsov-2021}, where it is attributed to \cite{Rogers_Pitman:1981}.)
 By combining \cite[Theorem 2.6 and Theorem  2.8]{Bryc-Kuznetsov-2021} %
 if  $Z_0$ has the law
 \[\mu(dx)%
 =\begin{cases}
 \frac{u^2-v^2}{ v}  e^{-u x} \sinh (v x) {\mathbf 1}_{\{x>0\}}d x & v\ne 0 \\
 u^2 x  e^{-u x}  {\mathbf 1}_{\{x>0\}}d x & v=0
 \end{cases}\]
 with $u-v>0$ and $u+v>0$ and
 $\gamma$ has exponential law $\nu(dx)=(u+v) e^{-(u+v)x} \ind_{x>0}dx$,  then  the  process $ \pp{2 \pp{\sup_{0\leq s\leq t} W_s\topp v-\gamma}_+-W_t\topp v}_{t\geq 0}$ has the same law as $\pp{Z_t-Z_0}_{t\ge 0}$.
 One can check that then the law of $Z_0$ and the law of $\gamma$ are related by   \eqref{density}.
 \arxiv{Indeed,
\[
     2 v e^{-2 v x}\int_x^\infty \frac{u^2-v^2}{2 v}\frac{e^{-(u-v)y}-e^{-(u+v)y}}{1-e^{-2 v y}} dy
    =(u^2-v^2) e^{-2 v x}\int_x^\infty e^{-(u-v)y} dy
    = \frac{(u^2-v^2)}{u-v} e^{-(u+v)x}.
\]
 So in this case, $\gamma$ is exponential with parameter $u+v$.
 }
Here we use local convergence established in Lemma \ref{L-loc} to  extend  this observation to more general initial laws.
\begin{proposition}\label{AK}
    Let  $(Z_t)_{t\ge 0}$ be a Markov process with a transition probability function \eqref{ZZ-trans} for some $v\in\RR$ and
an absolutely continuous initial distribution $\mu(dy)=f(y)dy$ on $\RR_+$.
    Let $\gamma$ be a random variable with the law \eqref{density}, such that $\gamma$ and $\pp{W_t\topp v}_{t\ge 0}$ are independent.
    Then the process
    \begin{equation}
        \label{conclude}
        \pp{2 \pp{\sup_{0\leq s\leq t} W_s\topp v-\gamma}_+-W_t\topp v}_{t\geq 0}
    \end{equation}
    has the same law as $\pp{Z_t-Z_0}_{t\ge 0}$.
\end{proposition}

The proof of  Proposition \ref{AK} is in Section \ref{Sect:proofAK}.

The paper is organized as follows. In Section \ref{PThm1} we  give a reconstruction formula for the generalized Pitman transform,
identify the laws in \eqref{Pit-rep} and prove Theorems \ref{Thm1} and \ref{Thm1-2}. The elementary technical proof of the reconstruction formula is deferred to the appendix. %
In Section \ref{T-YZ-proof} we prove Theorem \ref{T-YZ} and Proposition \ref{AK}. Section \ref{Sect:4} contains several additional observations: we give a continuous-time version of Theorem \ref{Thm1}, we relate \eqref{q-thin} to damage models, we
compute a Poisson example and give an example which shows that the same limit \eqref{Weak-lim+}  may arise from the
initial laws of $X\topp N_0$ that are not tight.

\section{Proofs of Theorem \ref{Thm1} and Theorem \ref{Thm1-2}}\label{PThm1}
For $m,n\in\ZZ$, we denote
$\Zin{m,n}=\{m,m+1,\dots,n\}$ and $\llbracket m,\infty)=\{m,m+1,\dots\}$. %
We write $a\wedge b$ for $\min\{a,b\}$, and $a\vee b$ for $\max\{a,b\}$.

\subsection{The laws}
The equality in  \eqref{Pit-rep} is a statement about the equality of the finite-dimensional distributions of two discrete-time processes.
We shall therefore compute the finite-dimensional laws of the pair of processes that appear in \eqref{Pit-rep} and determine when they are equal.

 For  $t\in \llbracket 0,\infty)$ we consider the following discrete analog of the space of continuous functions  on the interval $[0,t]$
vanishing at $0$:
\[C\Zin{0,t}:=\ccbb{(x_j)_{0\leq j\leq t}: x_0=0,\quad x_j-x_{j-1}\in\{-1,0,1\},\quad 1\leq j\leq t }.\]
For $\vv x\in C\Zin{0,t}$, we write $\vv x=(x_0,x_1,\dots,x_t)$.

 We will need the following functions on $C\Zin{0,t}$: the   backwards running   minima   $K_j$, $0\leq  j\leq t$, and the count $H$ of the horizontal steps,
 \begin{equation}\label{K-def}
   K_j(\vv x)=\min_{ i\in\Zin{j, t}}x_i, \;\quad  H(\vv x)=\#\{ i\in\Zin{1,t}: x_i=x_{i-1}\}.
 \end{equation}
 It is clear that $K_0$ takes values in $\{0,-1,\dots,-t\}$ and $H\in\{0,1,\dots,t\}$.

  We begin with the left-hand side of \eqref{Pit-rep}.
 \begin{proposition}With $\vv x\in C\Zin{0,t}$,  $K:=K_0(\vv x)$, $H:=H(\vv x)$ for $\rho>0$ we have \begin{equation}
     \label{law-X}
     \Pr(X_1-X_0=x_1,\dots,X_t-X_0=x_t)=
     \tfrac{\sigma^{H}}{(\rho+1/\rho+\sigma)^t\rho^{x_t}}\sum_{k=-K}^\infty \Pr(X_0=k)
    \tfrac{[x_t+k+1]_{\rho^2}}{[k+1]_{\rho^2}}.
        \end{equation}
 \end{proposition}
 \begin{proof}
   Since $X_j\geq0$, we  observe that the probability on the left-hand side of
    \eqref{law-X} is non-zero only when $X_0$ takes values in the set $\{-K,-K+1,\dots\}$.
    Thus \eqref{Z-trans} together with \eqref{xi-gen} yield
    \begin{multline}\label{2.3}\Pr(X_1-X_0=x_1,\dots,X_t-X_0=x_t)
     =\sum_{k=-K}^\infty \Pr(X_0=k)
    \prod_{j=1}^t \Pr(X_j=x_j+k|X_{j-1}=x_{j-1}+k)
    \\=\tfrac{\sigma^{H}}{(\rho+1/\rho+\sigma)^t}\sum_{k=-K}^\infty \Pr(X_0=k)
\prod_{j=1}^t \tfrac{\rho^{x_{j-1}}[x_j+k+1]_{\rho^2}}{\rho^{x_j}[x_{j-1}+k+1]_{\rho^2}}
    =\tfrac{\sigma^{H}}{(\rho+1/\rho+\sigma)^t\rho^{x_t}}\sum_{k=-K}^\infty \Pr(X_0=k)
    \tfrac{[x_t+k+1]_{\rho^2}}{[k+1]_{\rho^2}}.
    \end{multline}
 \end{proof}
  Next, we consider the right-hand side of \eqref{Pit-rep}.
Recall that $(S_t)$ is the random walk with steps \eqref{xi-gen},  $M_t=\max_{j\in\Zin{0,t}}S_j$ and $G$ is an independent random variable.

 \begin{proposition}\label{P-M-law}
 With $\vv x\in C\Zin{0,t}$,  $K:=K_0(\vv x)$, $H:=H(\vv x)$ we have
   \begin{multline}\label{2M-S-law}
     \Pr( 2(M_1-G)_+-S_1=x_1,\dots, 2(M_t-G)_+-S_t=x_t)
     \\
      =
     \tfrac{\sigma^{H}\rho^{x_t}}{(\rho+1/\rho+\sigma)^t}
       \pp{\Pr(G\geq -K) +\rho^{2(K-x_t)}[x_t-K]_{\rho^2} \Pr(G=-K)}.
   \end{multline}
 \end{proposition}

\begin{remark}\label{Rem-S}\sloppy
    Note that when $G\sim\mathrm{geo}(\rho^2)$, i.e., $\rho<1$,  then the right hand side in \eqref{2M-S-law}
     is ${\sigma^{H}\rho^{-x_t}}/{(\rho+1/\rho+\sigma)^t}$, %
    that is $\pp{2(M_t-G)_+-S_t}_{t\ge 0}%
    \stackrel{\calD}{=}\pp{ S_t}_{t\ge 0} $. %
    In fact, $\pp{2(M_t-G)_+-S_t}_{t\ge 0}\stackrel{\calD}{=}\pp{ S_t}_{t\ge 0}$ if and only if $G\sim\mathrm{geo}(\rho^2)$. Compare \cite[Corollary 3.14]{Croydon2023}. %

\end{remark}
\arxiv{Here is the proof of the equivalence. Suppose $\pp{2(M_t-G)_+-S_t}_{t\ge 0}\stackrel{\calD}{=}\pp{ S_t}_{t\ge 0}$. Then for $\rho\ne1$ we have
\[
\frac{\sigma^{H}\rho^{x_t}}{(\rho+1/\rho+\sigma)^t}
       \pp{\Pr(G\geq -K)+\rho^{2 (K-x_t)}[x_t-K]_{\rho^2}\,\Pr(G=-K)}=\frac{\sigma^{H}\rho^{-x_t}}{(\rho+1/\rho+\sigma)^t}.
\]
Taking sufficiently large $t$ above we conclude that for any $j=0,1,\ldots$ and any $x\in \Z$
\[
\Pr(G>j)+\rho^{-2(j+x)}[j+x+1]_{\rho^2}\Pr(G=j)=\rho^{-2x},
\]
where $q=\rho^2$. Subtracting side-wise two such equations for $j$ and $j-1$,  on noting that
$\rho^{-2(j+x)}[j+x+1]_{\rho^2}-1=\rho^{-2(j+x)}[j+x]_{\rho^2}$,  we arrive at
\[
\Pr(G=j)=\rho^2\Pr(G=j-1),\quad j=1,2,\ldots .
\]
}

 The proof of Proposition \ref{P-M-law}  relies  on the analysis of the mapping $\TT:\ZZ_{\geq 0}\times C\Zin{0,t}\to C\Zin{0,t}$, which is defined for $g\in\ZZ_{\geq 0}$ and $\vv s\in C\Zin{0,t}$   by the expression
\begin{equation}\label{def-T}
  \TT_g(\vv s):=\TT(g,\vv s) =\pp{2 \pp{\max_{i\in\Zin{0, j}} s_i-g}_+-s_j}_{0\leq j\leq t}.
\end{equation}
With $\vv S=(0,S_1,\dots,S_t)$, the left-hand side of \eqref{2M-S-law} is  $\Pr(\TT_G(\vv S)=\vv x)$.

We will also need two more functions on $C\Zin{0,t}$  that count the number of upward steps and downward steps:
\begin{equation}\label{UD-steps}
  U(\vv x)=\#\{j\in\Zin{1,t}: x_j-x_{j-1}=1\}, \quad   D(\vv x)=\#\{j\in\Zin{1,t}: x_j-x_{j-1}=-1\}.
\end{equation}
 We note the following obvious relations:
 \begin{equation*}
   U+D+H=t,\; U(\vv x)-D(\vv x)=x_t.
 \end{equation*}

It is clear that $\TT_g$ maps  $C\Zin{0,t}$ into itself and it turns out that the map
$\TT: \ZZ_{\geq 0}\times C\Zin{0,t}\to C\Zin{0,t}$ is onto. The main ingredient of our proof is the explicit description of the inverse images under $\TT$.

 It is well known that the Pitman transform $\TT_0$ appearing in Pitman's Theorem can be inverted using additional
information. In particular, a reconstruction result \cite[Proposition 2.2 (iv)]{biane2005littelmann} shows how to recover $\vv s$ from its Pitman transform image $\vv x=\TT_0(\vv s)$, with additional information about $\vv s$ given by the value of $r=\max \vv s$. The reconstruction formula becomes
$$\vv s(u)=2 \min\{r,\min_{v\in\Zin{u,t}}\vv x(v)\}-\vv x(u),$$
compare \cite[formula (6.3)]{herent2024discretetime}.
Recalling \eqref {K-def} and notation  $a\wedge b=\min\{a,b\}$, this reconstruction formula can be written as
\begin{equation}
    \label{rec0}
    s_j= 2 (r\wedge K_j(\vv x)) - x_j.
\end{equation}

With our $\TT$, which is  a mapping
$\ZZ_{\geq 0}\times C\Zin{0,t}\to C\Zin{0,t}$ given by formula \eqref{def-T}, our reconstruction formula uses additional information given by a slightly more complicated expression \eqref{s2r+} for
$r$, which depends   not only on $\max \vv s$   but also on $\vv x$.   (The special property of the Pitman transform $\TT_0$ that makes the reconstruction theorem simpler is positivity of the image, so that if $\vv x=\TT_0(\vv s)$, then $K_0(\vv x)=0$ and \eqref{s2r+} reduces to $r=\max \vv s$.)

For $r\in\Zin{K_0(\vv x),x_t}$, %
define $\vv s\topp r=\vv s\topp r(\vv x)\in C\Zin{0,t}$ by

 \begin{equation}
   \label{sr}
   s\topp r_j=2(r\wedge K_j(\vv x) -K_0(\vv x))-x_j,\quad \mbox{ where }     j\in\Zin{0,t}.
 \end{equation}
 Clearly, \eqref{sr} reduces to formula \eqref{rec0} when $K_0(\vv x)=0$.  Note that   formula   \eqref{sr} gives $s_0\topp r=0$ as well as $s\topp r_j-s\topp r_{j-1}\in\{-1,0,1\}$,
 so $\vv s\topp r$ is indeed in $C\Zin{0,t}$.

The proof of  our next result  gives the following reconstruction formula.
If $\vv s$ is such that $\vv x=\TT_g(\vv s)$ for some $g\in\ZZ_{\geq 0}$ and
\begin{equation}
    \label{s2r+}
    r:=\max \vv (\vv s+K_0(\vv x))_++K_0(\vv x)=\begin{cases}
  \max \vv s+2 K_0(\vv x)  & \max \vv s> -K_0(\vv x) \\
  K_0(\vv x) & \max \vv s\leq -K_0(\vv x),
\end{cases}
\end{equation}
then $\vv s=\vv s \topp r$ is given by \eqref{sr}, and there are two cases:
\begin{enumerate}[(i)]
    \item If $r>K_0(\vv x)$, then $g=-K_0(\vv x)$ is unique.
    \item If $r=K_0(\vv x)$, then $\vv x=\TT_g(\vv s)$ for any $g$ from the infinite set $\{ -K_0(\vv x), -K_0(\vv x)+1,\dots \}$.

\end{enumerate}
 \color{black}
\begin{proposition}\label{Prop-T}
\begin{enumerate}[(i)]
\item The mapping
$\TT: \ZZ_{\geq 0}\times C\Zin{0,t}\to C\Zin{0,t}$, as defined in \eqref{def-T}, is surjective (onto).
\item  The  inverse image  of  $\vv x\in C\Zin{0,t}$ under transformation $\TT$ is
\begin{equation}\label{Sets}
\ccbb{ \left(-K_0(\vv x),\vv s\topp r(\vv x)\right):
r\in\Zin{K_0(\vv x),  x_t}}\cup
\ccbb{ \left(g,\vv s\topp {K_0(\vv x)}(\vv x)\right): g\in\llbracket -K_0(\vv x)+1,\infty)},
\end{equation}
where
$\vv s\topp r(\vv x)$, $r\in \Zin{K_0(\vv x),x_t}$,    are distinct and given by \eqref{sr}.
\item   For $H$, $U$, $D$ defined in \eqref{K-def} and \eqref{UD-steps}, and   for $r\in\Zin{K_0(\vv x),x_t}$, we have
\begin{equation}\label{UDH2UDH}
H(\vv s\topp r (\vv x))=H(\vv x),\quad    U(\vv s\topp r (\vv x))= r-K_0(\vv x)+ D(\vv x),  \quad
 D(\vv s\topp r (\vv x))= K_0(\vv x)-r+U(\vv x).
\end{equation}
\end{enumerate}
\end{proposition}
The proof of Proposition \ref{Prop-T} is in Appendix \ref{sec:Prop-T}.

The following can be interpreted as a discrete version of Girsanov's theorem.
\begin{remark}\label{TBGW}
Fix $t\geq 1$ and $\sigma\geq 0$. Denote by ${\vv S}\topp {\rho}$ the (finite) random walk  $(0,S_1,\dots,S_t)$ with increments \eqref{xi-gen}. Then the law $\mathbb{P}_\rho$  of ${\vv S} \topp {\rho}$  is absolutely continuous with respect to the law of (the symmetric random walk) $\vv S\topp 1$ with the Radon-Nikodym derivative
\begin{equation}
    \label{Girsanov}
    \frac{d\mathbb{P}_\rho}{d\mathbb{P}_1}(\vv x)= \mathsf C \rho^{-x_t}, \quad \vv x\in C\Zin{0,t}.
\end{equation}

The normalizing constant is explicit: $\mathsf C=\frac{(2+\sigma)^t}{(\rho+1/\rho+\sigma)^t}$.
(For a more general case compare \cite[pg.  44]{lawler2010random})
\end{remark}

\begin{proof}[Proof of Remark \ref{TBGW}] We note that for $\vv x\in C\Zin{0,t}$,
  \begin{equation}\label{reg-G-1}
  \Pr(\vv S\topp \rho= \vv x)= \frac{ \sigma^{H(\vv x)} \left(1/\rho\right)^{U(\vv x)}\rho^{D(\vv x)}}{(\rho+1/\rho+\sigma)^t} =\frac{\sigma^{H(\vv x)}\rho^{-x_t}}{(\rho+1/\rho+\sigma)^t},
  \end{equation}
    which we use with $\rho>0$ and then with $\rho=1$.
\end{proof}
\begin{proof}[Proof of Proposition \ref{P-M-law}]
We use \eqref{UDH2UDH} and \eqref{Sets} of Proposition \ref{Prop-T}. We note that \eqref{UDH2UDH} gives
$$
H(\vv s\topp r)=H(\vv x)
\mbox{ and }
\vv s\topp r _t=2(r-K_0)-x_t.$$
Thus  with $\vv S=(S_0,\dots,S_t)$ we
  compute %
  \begin{multline*}
     \Pr\pp{\TT_G(\vv S)=\vv x}=
  \Pr\pp{(G,\vv S)\in \TT^{-1}(\{\vv x\}) }
  \\\stackrel{\eqref{Sets}}{ = }\Pr(\vv S=-\vv x)\sum_{g=-K_0}^\infty \Pr(G=g)+
  \Pr(G=-K_0)\sum_{r=K_0+1}^{x_t} \Pr(\vv S=\vv s \topp r (\vv x))
 \\
 \stackrel{\eqref{reg-G-1}}{=}\Pr(G\geq -K_0) \tfrac{\sigma^{H(\vv x)} \rho^{x_t}}{(\rho+1/\rho+\sigma)^t}
 +\Pr(G=-K_0)\sum_{r=K_0+1}^{x_t}
 \tfrac{\sigma^{H(\vv x)}  \rho^{x_t-2(r-K_0)}}{(\rho+1/\rho+\sigma)^t}
 \\=
\tfrac{\sigma^{H(\vv x)} \rho^{x_t}}{(\rho+1/\rho+\sigma)^t}\pp{\Pr(G\geq -K_0)
 +\Pr(G=-K_0)\sum_{r=K_0+1}^{x_t}
  \rho^{2(K_0-r)}}
   \\=
 \tfrac{\sigma^H \rho^{x_t}}{(\rho+1/\rho+\sigma)^t}\pp{\Pr(G\geq -K_0)
 +\rho^{2(K_0-x_t)}[x_t-K_0]_{\rho^2}\Pr(G=-K_0)}.
 \end{multline*}
\end{proof}
\subsection{Proofs of Theorems \ref{Thm1} and \ref{Thm1-2}}
 \begin{proof}[Proof of Theorem \ref{Thm1}: (i)$\Rightarrow$(ii)]

 We prove by induction with respect to $n$ that \eqref{q-damage} holds.

  By assumption, for every $t=1,2,\dots$ and every $\vv x\in C\Zin{0,t}$, the right hand side of
  \eqref{law-X} is the same as the right hand side of \eqref{2M-S-law}.
  Therefore,  for every $K\leq 0$ and $x_t\geq K$  we have  \begin{equation}\label{eq-K}
         \rho^{2x_t}\pp{\Pr(G\geq -K)
 +\rho^{2(K_0-x_t)}[x_t-K_0]_{\rho^2}\,\Pr(G=-K)}=\sum_{k=-K}^\infty \Pr(X_0=k)\tfrac{[x_t+k+1]_{\rho^2}}{[k+1]_{\rho^2}}.
  \end{equation}
  With $K=0$ and $x_t\ne 0$, in view of the identity $[x_t+k+1]_{\rho^2}-\rho^{2x_t}[k+1]_{\rho^2}=[x_t]_{\rho^2}$ we get
$$
 \Pr(G=0)
 =\frac{1}{[x_t]_{\rho^2}}\sum_{k=0}^\infty \Pr(X_0=k) \pp{\tfrac{[x_t+k+1]_{\rho^2}}{[k+1]_{\rho^2}} -\rho^{2x_t}}
=\sum_{k=0}^\infty  \
\tfrac{\Pr(X_0=k)}{[k+1]_{\rho^2}}.
$$
  Suppose \eqref{q-damage} holds for some $n\geq 0$.
Note that  for every $K\leq 0$ we can find $t\geq 1$ and  $\vv x\in C\Zin{0,t}$
  such that $K_0(\vv x)=K$ and {$x_t=0$}.
 (For example, take $\vv x=(0,0)$ and $t=1$ to get  $K=0$ and to get $K<0$ we can take $t=-2K$, $x_j=-j$
 for $j\in \{0,1,\ldots,-K\}$ and $x_j=2K+j$ for $j\in\{-K+1,-K+2,\ldots,2K\}$).
  Thus from \eqref{eq-K} for every $K\leq 0$ we get
   \begin{equation}\label{***1}
       \Pr(X_0\geq -K)
     = \Pr(G\geq -K)+\rho^{2K}[-K]_{\rho^2}\Pr(G=-K).
  \end{equation}
  Taking the difference of two  identities \eqref{***1} for the consecutive values of $K=-n$ and $K=-n-1$ we get
  \begin{equation*}
    \Pr(X_0=n)
    = \Pr(G=n)+\rho^{-2n}[n]_{\rho^2}\Pr(G=n)
    -\rho^{-2(n+1)}[n+1]_{\rho^2}\Pr(G=n+1).
  \end{equation*}
{
That is   \begin{equation*}
    \Pr(X_0=n)
    =  \pp{\rho^2\Pr(G=n)
    -\Pr(G=n+1)}\rho^{-2(n+1)}[n+1]_{\rho^2}.
  \end{equation*}
}

Thus by induction assumption
\begin{multline*}
 \Pr(G=n+1)=\rho^2 \Pr(G=n)-\tfrac{\rho^{2n+2}\Pr(X_0=n)}{[n+1]_{\rho^2}}=\rho^{2(n+1)} \sum_{m=n}^\infty \tfrac{\Pr(X_0=m)}{[m+1]_{\rho^2}}-\tfrac{\rho^{2n+2}\Pr(X_0=n)}{[n+1]_{\rho^2}}
 \\=  \rho^{2(n+1)}\sum_{m=n+1}^\infty \tfrac{\Pr(X_0=m)}{[m+1]_{\rho^2}}.
\end{multline*}
This ends the proof by induction.
  \color{black}

 \end{proof}
  \begin{proof}[Proof of Theorem \ref{Thm1}: (ii)$\Rightarrow$(i)]
   Suppose that the law of $G$ is \eqref{q-damage}.
     Taking into account formulas
  \eqref{law-X} and \eqref{2M-S-law}, we need to show that for every $K\leq 0$ and $x_t\geq K$ we have
  \begin{equation}
      \label{aux}
 \Pr(G\geq -K) +\rho^{2(K-x_t)}[x_t-K]_{\rho^2} \Pr(G=-K)  =  \rho^{-2x_t}\sum_{k=-K}^\infty \Pr(X_0=k)
    \tfrac{[x_t+k+1]_{\rho^2}}{[k+1]_{\rho^2}}.
  \end{equation}
in view of \eqref{q-damage} the left-hand side of \eqref{aux} is     \begin{multline*}
         \sum_{j=-K}^\infty \sum_{m=j} ^\infty \tfrac{\rho^{2j}}{[m+1]_{\rho^2}}\Pr(X_0=m)
 +\rho^{-2K}\,\sum_{m=-K}^\infty \tfrac{\Pr(X_0=m)}{[m+1]_{\rho^2}}\,\rho^{2(K-x_t)}[x_t-K]_{\rho^2}
\\ =
        \sum_{m=-K}^\infty \tfrac{\Pr(X_0=m)}{[m+1]_{q}}\pp{\sum_{j=-K}^m  \rho^{2j}+\rho^{-2x_t}[x_t-K]_{\rho^2}
}
  \\  =
       \rho^{-2x_t} \sum_{m=-K}^\infty \tfrac{\Pr(X_0=m)}{[m+1]_{q}} \left(\rho^{2x_t}[m+1]_{\rho^2}-\rho^{2x_t}[-K]_{\rho^2}+[x_t-K]_{\rho^2}\right).
  \end{multline*}
  Thus \eqref{aux} follows since
  $$
  [x_t-K]_{\rho^2}-\rho^{2x_t}[-K]_{\rho^2}=[x_t]_{\rho^2}\quad\mbox{and}\quad  \rho^{2x_t}[m+1]_{\rho^2}+[x_t]_{\rho^2}=[x_t+m+1]_{\rho^2}.
  $$
 \end{proof}

  \begin{proof}[Proof of Theorem \ref{Thm1-2}]

If $\rho<1$, then $-\inf_{u\ge 0}S_u\stackrel{\mathcal D}{=}-\inf_{u\ge t}(S_u-S_t)\stackrel{\mathcal D}{=}G\sim \rm{geo}(\rho^2)$.

 For any $t\in\{0,1,\ldots\}$ with $\vv S=(0,S_1,\dots,S_t)$ and $\vv x=(x_0,\dots,x_t)\in C\Zin{0,t}$
we have
\begin{multline*}
\Pr(\vv S =\vv x,\,\inf_{0\leq u<\infty}  S_u+V\geq0)=\Pr(\vv S =\vv x, V\ge -K,\, \inf_{u\geq t} (S_u-S_t)\geq -V-x_t)\\
=\Pr(\vv S =\vv x)\Pr(V\ge -K,\inf_{u\geq t} (S_u-S_t)\geq -V-x_t)
\\ =\tfrac{\rho^{-x_t}\sigma^{H}}{(\rho+1/\rho+\sigma)^t}\sum_{j=-K}^\infty
\Pr(V=j)(1-\rho^{2(j+x_t+1)})
=\tfrac{(1-\rho^2)\rho^{-x_t}\sigma^{H}}{(\rho+1/\rho+\sigma)^t}\sum_{j=-K}^\infty
\Pr(V=j)[j+x_t+1]_{\rho^2},
\end{multline*}
where $K=\min_{j\in\Zin{0,t}}x_j$, $H=\#\{j\in \Zin{1,t}: x_j=x_{j-1}\}$. Since for $t=0$ the above formula gives
\begin{equation*}
    \label{S+V}
    \Pr(\inf_{0\leq u<\infty}  S_u+V\geq0) = \ {(1-\rho^2)} c,
\end{equation*}
where $c=\sum_{j=0}^\infty\,\Pr(V=j)[j+1]_{\rho^2}$, we get
\begin{equation}
    \label{S+V2}\Pr\left(\vv S =\vv x\left|\inf_{0\leq u<\infty}  S_u+V\geq0\right.\right)=\tfrac{\rho^{-x_t}\sigma^{H}}{c(\rho+1/\rho+\sigma)^t}\sum_{j=-K}^\infty
\Pr(V=j)[j+x_t+1]_{\rho^2}.
\end{equation}

Taking into consideration \eqref{S+V2} and \eqref{2.3} we see that (i) is equivalent to
\begin{equation}\label{V_eq}
\sum_{j=-K}^\infty
\Pr(V=j)[j+x_t+1]_{\rho^2}=c\sum_{j=-K}^\infty
\Pr(X_0=j)\tfrac{[j+x_t+1]_{\rho^2}}{[j+1]_{\rho^2}}.
\end{equation}

Taking $t$ large enough in \eqref{V_eq} we conclude that for all $k\in \mathbb Z_{\ge0}$ and $x\ge k$ we have
\[
\sum_{j=k}^\infty
\Pr(V=j)[j+x+1]_{\rho^2}=c\sum_{j=k}^\infty
\Pr(X_0=j)\tfrac{[j+x+1]_{\rho^2}}{[j+1]_{\rho^2}}.
\]

Now, if (i) holds, we subtract sidewise the above equality for $k$ and $k+1$. Consequently,
\[
\Pr(V=k)=c\tfrac{\Pr(X_0=k)}{[k+1]_{\rho^2}},\quad k\in\mathbb Z_{\ge0},
\]
which proves (ii).

On the other hand if (ii) holds, then we have \eqref{V_eq}, and thus (i) is immediate.

 \end{proof}
\section{Proof of Theorem \ref{T-YZ} and Proposition \ref{AK}}\label{T-YZ-proof}

We will use the following fact which follows from the proof of
\cite[Theorem 5.5]{Billingsley1968}.

\begin{lemma}\label{Bill-T5.5}
Suppose that a sequence $\mu_N$ of probability measures on  $\RR_+$  converges to a probability measure $\mu$. For $L>0$ let $\Psi,\Psi_N:\RR_+\to [-L,L]$, $N=1,2,\ldots$, be measurable functions such that for all $u_N\to  u\in\R_+$ we have
\[\limn \Psi_{N}( u_N) = \Psi( u).%
\]
Then
\begin{equation*}\label{Bil:Ex6.6}
\lim_{N\to\infty} \int \Psi_N( u)\mu_N(d u)= \int \Psi( u)\mu(d  u).
\end{equation*}
\end{lemma}
\arxiv{\begin{proof}
By \cite[Theorem 5.5]{Billingsley1968}, see also \cite[Exercise 6.6]{billingsley99convergence},
for every bounded continuous function $f$ on $\RR$ we have
\[\int f(\Psi_N( u))\mu_N(d  u)\to \int f(\Psi(u))\mu(d u).\]
Now take a bounded continuous function such that $f(x)=x$ on $[-L,L]$.
\end{proof}}

For a probability measure $\mu$ on $\RR_+:=[0,\infty)$ and   $x>0$, define
\begin{equation}\label{cont-unif}
   F_\mu\topp v(x)=\mu[0,x]+\frac{1}{g_v(x)}\int_{(x,\infty)}g_v(y) dy.
\end{equation}

\begin{lemma}\label{Lem-F} Function $F_\mu\topp v$ is  non-negative increasing.   It is differentiable at every point of continuity  of $\mu$ in $(0,\infty)$      with the derivative $\frac{d F_\mu \topp v(x)}{dx}=f_\mu \topp v(x)$ given by \eqref{density}.
Furthermore, $F(0+):=\lim_{x\searrow 0} F_\mu\topp v(x)$ exists, $F(0+)\ge 0$, and  $\lim_{x\to\infty}    F_\mu\topp v(x)=1$.
\end{lemma}
We set $F_\mu\topp v(x)=0$ for $x<0$ and define $F_\mu\topp v(0)$ as the right-hand side limit. Then $F_\mu\topp v$ is a cumulative distribution function of a probability measure with density on $(0,\infty)$ and a possible atom at $0$ with mass $F_+(0)=1-\int_0^\infty f_\mu\topp v(x)dx$.
\begin{proof}[Proof of Lemma \ref{Lem-F}]
From \eqref{cont-unif} it is clear that $F_\mu\topp v(x)\geq 0$ as both terms are nonnegative.

In the proof we use several times the fact that $g_v$ is strictly decreasing on $(0,\infty)$ for all $v\in\RR$, as $g_v'(x)=- g_v^2(x)e^{-2 v  x}$.

    If $0<x_1<x_2$ then in view of \eqref{cont-unif} we have
    \begin{multline}\label{Diff-F}
   F_\mu\topp v(x_2)- F_\mu\topp v(x_1)\\= \left(\frac{1}{g_v(x_2)}-\frac{1}{g_v(x_1)}\right)\int_{(x_2,\infty)}g_v(y)\mu(dy)+
   \int_{(x_1,x_2)}\left(1-\frac{g_v(y)}{g_v(x_1)}\right)\mu(dy)+\left(1-\frac{g_v(x_2)}{g_v(x_1)}\right)\mu(\{x_2\}).
\end{multline}
Thus,  it is clear that $F_\mu\topp v(x_2)- F_\mu\topp v(x_1)\geq 0$ and furthermore that $F_\mu\topp v$ is continuous on  $(0,\infty)$.

 From \eqref{Diff-F} we compute the derivative at $x_2>0$ by computing
 \[\lim_{x_1\nearrow x_2} \tfrac{F_\mu\topp v(x_2)- F_\mu\topp v(x_1)}{x_2-x_1}
 \mbox{ and } \lim_{x_1\searrow x_2} \tfrac{F_\mu\topp v(x_2)- F_\mu\topp v(x_1)}{x_2-x_1}.\]
 It is clear that  the first term in  \eqref{Diff-F} gives
 $$
 - \frac{g_v'(x_2)}{g_v^2(x_2)}\int_{(x_2,\infty)}g_v(y)\mu(dy)   =e^{-2 v x_2}\int_{(x_2,\infty)}g_v(y)\mu(dy).
 $$
For the  second term
\begin{multline*}
    0\leq \frac{1}{x_2-x_1}\int_{(x_1,x_2)}\left(1-\frac{g_v(y)}{g_v(x_1)}\right)\mu(dy)
    \\ \leq
  \frac{1}{x_2-x_1}\int_{(x_1,x_2)}\left(1-\frac{g_v(x_2)}{g_v(x_1)}\right)\mu(dy)=
 \frac{g_v(x_1)-g_v(x_2)}{(x_2-x_1)g_v(x_1)}\mu((x_1,x_2))\to 0.
\end{multline*}

The third term converges to $-\frac{g'_v(x_2)}{g_v(x_2)}\mu(\{x_2\})$.
   Adding the two contributing terms we get
 \[\lim_{x_1\nearrow x_2} \tfrac{F_\mu\topp v(x_2)- F_\mu\topp v(x_1)}{x_2-x_1}= - \frac{g_v'(x_2)}{g_v^2(x_2)}\int_{[x_2,\infty)}g_v(y)\mu(dy)
 =e^{-2 v x_2}\int_{[x_2,\infty)}g_v(y)\mu(dy).
 \]
    Reversing the roles of $x_1,x_2$ from the same argument we get

      \[\lim_{x_1\searrow x_2} \tfrac{F_\mu\topp v(x_2)- F_\mu\topp v(x_1)}{x_2-x_1}=e^{-2 v x_2}\int_{[x_2,\infty)}g_v(y)\mu(dy).\]

To prove the last part of the conclusion, we note that
\[ 0\leq  \int_{(x,\infty)} \frac{g_v(y)}{g_v(x)}\mu(dy)\leq \mu(x,\infty)\to 0 \quad \mbox{as $x\to\infty$}.\]

Therefore, in view of \eqref{cont-unif}, we see that $\lim_{x\to\infty}    F_\mu\topp v(x)=1$.
\end{proof}

Next, we show that   convergence in distribution of $X_0\topp N/\sqrt{N}$ implies  convergence in distribution of $G\topp N/\sqrt{N}$  given by  \eqref{q-thin}.
\begin{lemma}[Continuity lemma] \label{lem-cont}  \sloppy Suppose $\rho_N=1- v/{\sqrt{N}}+o\left({1}/{\sqrt{N}}\right)$, $v\in\mathds R$, and let $G\topp N$  be given by  \eqref{q-damage}.    If
 \[\lim_{N\to\infty} \Pr\left(\tfrac{X_0^{(N)}}{\sqrt{N}}\leq x\right) =\mu[0,x]\]
for every point $x> 0$ of continuity of a probability measure $\mu$, then
\[\lim_{N\to\infty}\Pr\left(\tfrac{G\topp N}{\sqrt{N}}\leq x\right)= F_\mu\topp v(x)\] for all $x>0$, with  $F_\mu\topp v$ defined by \eqref{cont-unif}.
\end{lemma}

\begin{proof} Denote by $\mu_N$ the law of $X_0\topp N/\sqrt{N}$. Let $h:\RR_+\to[0,1]$ be a continuous function
such that $h(x)=0$ for $x\leq \delta$ for some $\delta>0$.

 By \eqref{q-thin} with $\rho=\rho_N$, we have
\begin{multline*}
    \EE\bb{h\pp{\tfrac{G\topp N}{\sqrt{N}}}}
=\EE\bb{\EE\bb{h\pp{\tfrac{G\topp N}{\sqrt{N}}}\middle| X_0\topp N}}
=\EE\bb{\frac{1}{[X_0\topp N+1]_{\rho_N^2}}\sum_{k=0}^{X_0\topp N}\,\rho_N^{2k}\, h\pp{\tfrac{k}{\sqrt{N}} }}
\\=\int_{\RR_+}\frac{\sqrt{N}}{[1+y\sqrt{N}]_{\rho_N^2}}\pp{\tfrac{1}{\sqrt{N}}\sum_{0\leq k\leq y\sqrt{N}}\rho_N^{2k} h\pp{\tfrac{k}{\sqrt{N}}}}\mu_N(dy)
=
\int_{\RR_+}\Psi_N(y)\mu_N(dy),
\end{multline*}
with  %
\begin{equation}\label{Psi-N}
    \Psi_N(y)=\tfrac{\sqrt{N}}{\left[1+y \sqrt{N}\right]_{\rho_N^2}}\int_0^y\psi_N(x)dx,
\end{equation}
where, using $h=0$ on $[0,\delta]$, for  $N>1/\delta^2$ we have
\begin{equation}\label{psi-N}
 \psi_N(x)=\sum_{1\leq  k\leq y\sqrt{N}}  \rho_N^{2k}   h\pp{\tfrac{k}{\sqrt{N}}}\ind_{\left(\tfrac{k-1}{\sqrt{N}},\tfrac{k}{\sqrt{N}}\right]}(x).
\end{equation}
With $k_N(x):=\ceil{ x \sqrt{N}}$ we have
$
\psi_N(x)= \rho_N^{2k_N(x)}  h\pp{{k_N(x)}/{\sqrt{N}}}\to e^{-2 v x} h(x)
$ for $x\in[0,y)$.
Since
$|\psi_N(x)|\leq C(y)=\|h\|_\infty e^{y(|v|+1)}$ for all $N$ large enough, using Lebesgue's dominated convergence theorem we get
\[
\lim_{N\to \infty}\int_0^y\psi_N(x)dx = \int_0^y e^{-2 v x} h(x) dx.
\]
Since   $$\tfrac{\sqrt{N}}{\left[1+y \sqrt{N}\right]_{\rho_N^2}}%
\to g_v(y),$$
therefore
\begin{equation*}
    \lim_{N\to\infty}\Psi_N(y)=\Psi(y):= g_v(y)\int_0^y e^{-2 v x} h(x) dx,
\end{equation*}
where we define $\Psi(0)=0$ by continuity (in fact, $\Psi(y)=0$ for  $y\in(0,\delta)$).

We now verify assumptions of  Lemma \ref{Bill-T5.5}.
From \eqref{psi-N}  we see that
$$
0\leq \int_0^y \psi_N(x)dx\leq  \frac{1}{\sqrt{N} }\|h\|_\infty \sum_{0\leq k \leq y\sqrt{N}}\rho_N^{2k} = \frac{1}{\sqrt{N}}\|h\|_\infty \left[1+\floor{y\sqrt{N}}\right]_{\rho_N^2} .
$$
Note that for any $a \ne 1$, $a>0$,  if $x\leq y$ then $(1-a^x)/(1-a^y)\leq 1$. Therefore, from \eqref{Psi-N}  we see that
 \[0\leq \Psi_N(y) \leq \frac{\left[1+\floor{y\sqrt{N}}\right]_{\rho_N^2}}{[1+y\sqrt{N}]_{\rho_N^2}}\|h\|_\infty\leq \|h\|_\infty.\]
Since $h=0$ on $[0,\delta]$, if $y_N\to 0$ then
$\Psi_N(y_N)=0 $ for $N$ large enough and $\Psi(0)=0$.
If $y_N\to y>0$ then $\Psi_N(y_N)\to \Psi(y)$ because
the first factor in \eqref{Psi-N} converges, and
$|\int_{y_N}^y \psi_N(dx)|\leq C(y) |y_N-y|\to 0$.

Thus by Lemma \ref{Bill-T5.5},
\[ \lim_{N\to\infty} \EE\bb{h\pp{\frac{G\topp N}{\sqrt{N}}}}=
 \int_{[0,\infty)}
 \Psi(y) \mu(dy)
 =\int_{[0,\infty)}
 g_v(y)\int_0^y e^{-2 v x} h(x) dx\; \mu(dy)
 =\int_0^\infty h(x)f_\mu\topp v(x)dx.\]
 (In the last equality we used Fubini's theorem.)

  The last integral is equal to  $\int_{\RR_+} h(x)F_\mu\topp v(dx)$, thus the convergence holds also for $1-h$ instead of $h$.
 This removes the condition $h(x)=0$ on interval $[0,\delta]$. Thus   weak convergence follows by standard arguments of sandwiching $\ind_{[0,x]}$ for $x>0$ between two such continuous functions $h$ that are arbitrarily close in $L_1(\RR_+,F_\mu\topp v(dx))$.
\end{proof}

\begin{proof}[Proof of Theorem \ref{T-YZ}]
  Let  $x>0$ be a continuity point of $\mu$. The key step in the proof is  convergence
  \begin{equation} \label{Key}
\lim_{N\to \infty}    \Pr\pp{\tfrac{G\topp N}{\sqrt{N}} \leq x}=
F_\mu\topp v(x),
  \end{equation}
  which holds true by Lemma \ref{lem-cont}.

By Donsker's theorem applied to the triangular array  $\left\{(\xi_i^{(N)})_{i\ge 1}\right\}_{N\ge 1}$ of row-wise i.i.d. random variables with the law of $\xi_i^{(N)}$  as given in \eqref{xi-gen} with $\rho=\rho_N$ and  $S_n\topp N=\sum_{i=1}^n\xi_i\topp N$, we have
   \[
  \tfrac{1}{\sqrt{N \Var(\xi\topp N)}} \pp{S_{\floor{Nt}}\topp N -\floor{Nt}\EE(\xi\topp N) }_{t\geq 0}\toD \pp{W_t}_{t\geq 0}.
   \]
   By direct calculations,   we get
  \[\EE(\xi_i\topp N)=\tfrac{1-\rho_N^2}{\rho_N ^2+\rho_N  \sigma +1}=
\tfrac{2 v}{\sigma +2}\tfrac{1}{\sqrt{N}} +O\left(\tfrac{1}{N}\right),\]
\[\Var(\xi_i\topp N)=\tfrac{\rho_N  \left(\rho_N ^2 \sigma +4 \rho_N +\sigma \right)}{\left(\rho_N ^2+\rho_N  \sigma +1\right)^2}
   =\tfrac{2}{\sigma +2} +O\left(\tfrac{1}{N}\right).\]
 Consequently
  \[\tfrac{1}{\sqrt{N } }\pp{S_{\floor{Nt}}}_{t\geq 0}\toD \pp{\sqrt{\tfrac{2}{2+\sigma}}W_t+\tfrac{2}{2+\sigma}v t}_{t\geq 0} \stackrel{\calD}{=} \pp{B_t\topp v}_{t\geq 0},
\]
where $B_t\topp v=W_{2t/(2+\sigma)}\topp v$ with  $W_t\topp v=W_t+vt$ denoting the standard Wiener process with drift $v$.

      Applying  \eqref{Pit-rep}, since  \eqref{Key} is  weak convergence, we get
  \begin{multline*}
        \tfrac{1}{\sqrt{N}}\pp{X_{\floor {Nt}}\topp N-X_0\topp N}_{t\geq 0}
     \stackrel{\calD}{=}\tfrac{1}{\sqrt{N}}\pp{2\pp{\max_{s\leq t} S_{\floor{Nt}}\topp N-G\topp N}_+-S_{\floor{Nt}}\topp N}_{t\geq 0}
\\   \toD \pp{2 \pp{\sup_{0\leq s\leq t} B_s\topp v-\gamma}_+-B_t\topp v}_{t\geq 0}
  \end{multline*}
  and  the result follows.
\end{proof}

\subsection{Proof of Proposition \ref{AK}}\label{Sect:proofAK}
We prove Proposition \ref{AK} by providing a discrete approximation to $(Z_t-Z_0)$.
 To ameliorate parity issues,  in the next result we use the even integer function $\ffloor{x}:=2\floor{x/2}$.
 \begin{lemma}\label{L-loc}
 Suppose $\rho_N=1-v/\sqrt{N}+o(1/\sqrt{N})$ and $\sigma=0$.
 Then for $x>0$, $y>0$ we have
 \begin{equation}
  \label{KPZ:fixed:rho-trans}
  \lim_{N\to\infty}\sqrt{N}\Pr\Big(X_{\ffloor{t N}}=\ffloor{y \sqrt{N}}  \Big| X_{\ffloor{s N}}\topp N=\ffloor{x \sqrt{N}}\Big)= 2\; \g_{t-s}(x,y)
  \begin{cases}
    \tfrac{ \sinh(v y)}{ \sinh(v x)}\;e^{-v^2 (t-s)/2} & v\ne 0 \\
    \tfrac{ y}{ x}  & v=0,
\end{cases}
\end{equation}
 where $\g_t(x,y)$ is given by  \eqref{eta:heat}. %
 \end{lemma}
 \begin{proof}
It suffices to consider the case $s=0$.  From the calculations in \eqref{2.3} with $\sigma=0$ we see that
\[\Pr(X_1=x_1,\dots,X_t=x_t|X_0=x_0)
=\frac{\rho^{x_0-x_t}}{(\rho+1/\rho)^t} \frac{[x_t+1]_{\rho^2}}{[x_0+1]_{\rho^2}}.\]
does not depend on $x_1,\dots,x_{t-1}$.
Denote by $\calP$ the set of all Dyck paths $\vv x$ from $(0,x_0)$ to $(t,x_t)$ for
$t$ of the same parity as $x_t-x_0$.
Summing over all
paths in $\calP$  we get
\begin{multline*}
       \Pr(X_t=x_t|X_0=x_0)=  \#\calP \frac{\rho^{x_0-x_t}}{(\rho+1/\rho)^t} \frac{[x_t+1]_{\rho^2}}{[x_0+1]_{\rho^2}}
=
\frac{\binom{t}{(t+x_t-x_0)/2} -\binom{t}{(t+x_t+x_0+2)/2}}{(\rho+1/\rho)^t}\; \rho^{x_0-x_t} \frac{[x_t+1]_{\rho^2}}{[x_0+1]_{\rho^2}}.
\end{multline*}
We apply this formula to compute asymptotics of
$$  \sqrt{N}\Pr\Big(X_{\ffloor{t N}}\topp N=\ffloor{y \sqrt{N}}  \Big| X_{0}\topp N=\ffloor{x \sqrt{N}}\Big).$$
Since $\rho_N=1-v/\sqrt{N}+o(1/\sqrt{N})$ we have
\[
\lim_{N\to\infty}
\tfrac{1}{2^{\ffloor{t N}}}(\rho_N+1/\rho_N)^{\ffloor{t N}}=e^{t v^2/2}.
\]
By the Stirling formula
\[
\lim_{N\to\infty}\tfrac{\sqrt{N}}{2^{\ffloor{t N}}}
\begin{pmatrix}[1.05]
\ffloor{t N}
\\
\tfrac{\ffloor{t N}+\ffloor{y \sqrt{N}}-\ffloor{x \sqrt{N}}}{2}
\end{pmatrix}=\frac{\sqrt{2}}{\sqrt{\pi t}}e^{-\frac{(y-x)^2}{2t}}.
\]
Similarly,
\[\lim_{N\to\infty}\frac{\sqrt{N}}{2^{\ffloor{t N}}}
\begin{pmatrix}[1.05]
{\ffloor{t N}}
\\
{\frac{\ffloor{t N}+\ffloor{y \sqrt{N}}+\ffloor{x \sqrt{N}}+2}{2}}
\end{pmatrix}
=\frac{\sqrt{2}}{\sqrt{\pi t}}e^{-\tfrac{(y+x)^2}{2t}}.\]
Next we observe that
\[\lim_{N\to\infty} \rho_N^{\ffloor{x\sqrt{N}}-\ffloor{y\sqrt{N}}} = e^{v(y-x)}\]
and
\[
\lim_{N\to\infty} \frac{\left[ \ffloor{y\sqrt{N}}+1\right]_{\rho_N^2}}{\left[\ffloor{x\sqrt{N}}+1 \right]_{\rho_N^2}}
=\frac{g_v(x)}{g_v(y)},
\]
where  $g_v$ is given by \eqref{density}.
Thus
$$
\lim_{N\to\infty}  \sqrt{N}\Pr\Big(X_{\ffloor{t N}}\topp N=\ffloor{y \sqrt{N}}  \Big| X_{0}\topp N=\ffloor{x \sqrt{N}}\Big)
=2  e^{-t v^2/2} e^{v(y-x)}\frac{g_v(x)}{g_v(y)}\g_t(x,y).
$$
Since $g_0(x)=1/x$ and  $e^{v(y-x)}\frac{g_v(x)}{g_v(y)}= \tfrac{ \sinh(v y)}{ \sinh(v x)}$ if $v\ne 0$, this ends the proof.
 \end{proof}
{\begin{proof}[Proof of Proposition \ref{AK}]
 For $k\in\ZZ_{\geq 0}$ define
\[p_{2k}\topp N:= \int_{2k/\sqrt{N}}^{(2k+2)/\sqrt{N}}f(x)dx,\] so that
$\lim_{N\to\infty}\frac{\sqrt{N}}{2} p_{\ffloor{x\sqrt{N}}}\topp N=f(x)$ outside a set of Lebesgue measure 0. (For the Lebesgue density theorem, see, e.g. \cite[Theorem 7.10]{rudin1987real}
or \cite{Morayne1989} for a martingale proof.)

Consider a family of Markov processes $\pp{X_t\topp N}_{t=0,1,\ldots}$ with the initial law $\Pr(X_0\topp N=2k)=p_{2k}\topp N$, $k\in\ZZ_{\geq 0}$, and with transition probabilities \eqref{Z-trans} that depend on $N$ through $\rho=\rho_N=1-v/\sqrt{N}+o(1/\sqrt{N})$.

    For fixed $t_0=0<t_1<\dots<t_d$ and $x_0,\dots,x_d\geq 0$, by \eqref{KPZ:fixed:rho-trans} and Markov property we have
    \begin{multline*}
      \lim_{N\to\infty}\tfrac{ N^{(d+1)/2}}{2^{d+1}} \Pr\pp{\tfrac{X_0\topp N}{\sqrt{N}}=\tfrac{\ffloor{x_0\sqrt{N}}}{\sqrt{N}},
      \tfrac{X_{\ffloor{t_1N}}\topp N}{\sqrt{N}}=\tfrac{\ffloor{x_1\sqrt{N}}}{\sqrt{N}},\dots, \tfrac{X_{\ffloor{t_dN}}\topp N}{\sqrt{N}}=\tfrac{\ffloor{x_d\sqrt{N}}}{\sqrt{N}}}
\\= f(x_0)\prod_{j=1}^d \g_{t_j-t_{j-1}}(x_{j-1},x_j).
    \end{multline*}
By \cite[Theorem 3.3]{billingsley99convergence}, this implies
 convergence of finite dimensional distributions,
 \[N^{-1/2}\pp{X_{\ffloor{N t}}\topp N}_{t\in [0,\infty)}
 \fddto\pp{Z_t}_{t\in[0,\infty)} \mbox{ as $N\to\infty$.}\]
 To conclude the proof, we invoke Theorem \ref{T-YZ}    with $\sigma=0$. According to this theorem,
 $ N^{-1/2}\pp{X_{\ffloor{N t}}\topp N-X_0\topp N}_{t\in [0,\infty)}$ converges to
\eqref{conclude}.
\end{proof}
}
\section{Additional observations and remarks}\label{Sect:4}
 \subsection{Examples related to Theorem \ref{T-YZ}}
Here we show that the same limit \eqref{Weak-lim+} may arise from the initial laws of $X_0\topp N$ that are not tight.
  It is clear that if $X_0=n$   then $G$ defined through \eqref{q-thin} is supported on the set $\{0,1,\dots,n\}$. This leads to convergence in
Theorem \ref{T-YZ} with deterministic $X_0\topp N$.
Suppose $X_0\topp N =\floor{ N^{1/2+\eps}}$ with $\eps>0 $ while $\rho=1-v/\sqrt{N}$.  Then   as $N\to\infty$, \eqref{q-thin}  gives
\begin{equation*}
\Pr(G\topp N\leq x \sqrt{N})=
\tfrac{\rho^{2\floor{x \sqrt{N}}+1}-1}{\rho^{2\floor{N^{1/2+\eps}}+1}-1}
\to \begin{cases}
1-e^{-2 v x}  & \eps>0, v>0, x\geq 0, \\
\\
\tfrac{1-e^{-2 v x}}{1-e^{-2 v }} & \eps=0, v\ne 0,  x\in[0,1],
\\ \\
x &\eps=0,  v=0, x\in[0,1].

\end{cases}
\end{equation*}
Therefore, we
get
\begin{equation*}\label{Weak-lim*}
\frac{1}{\sqrt{N}}\pp{X_{\floor {N t}}\topp N-X_0\topp N}_{t\geq 0}\toD \pp{2 \pp{\sup_{0\leq s\leq t} B_s\topp v-\gamma}_+-B_t\topp v}_{t\geq 0} \mbox{ as $N\to\infty$},
\end{equation*}
where, depending on $\eps$,   the independent random variable $\gamma$ has  exponential, truncated exponential, or uniform density
\begin{equation*}\label{4.3}
    f_\gamma(x) = \begin{cases}
    2 v e^{-2vx}\ind_{x\geq 0}  & \eps>0, v>0,  \\
\\
\tfrac{2v}{1-e^{-2 v }}e^{-2 v x}\ind_{x\in[0,1]} & \eps=0, v\ne 0,
\\ \\
\ind_{x\in[0,1]} &\eps=0,  v=0.

\end{cases}
\end{equation*}
The case $\eps=0$ is covered by Theorem \ref{T-YZ} but the case $\eps>0$ is not, and  gives the same limit as \eqref{Weak-lim} with $u=v$.

\subsection{A comment on damage models}\label{sec:dmg}

  Let $(N,R)$ be a pair of nonnegative integer-valued random variables. Assume that the conditional distribution $\Pr_{R|N}$ with support $\{0,1,\ldots,N\}$ is specified by a function $(s(r|n))_{r\in\{0,\ldots,n\}}$, $n=0,1,\ldots$. Then $(N,R)$ is called the damage model generated by $s$, where $R$ represents the number of surviving observations when the original observation $N$ is subject to a destructive process, while $D=N-R$ is the damaged part of $N$.
  The popular damage model is when the survival distribution is binomial, that is, with $s(r|n)=\tbinom{n}{r}\,p^r(1-p)^{n-r}$, with $p\in(0,1)$. It is well known that in the binomial damage model if $N$ is a Poisson random variable then $R$ and $D$ are independent and both have Poisson distributions. It is also known that if in the binomial damage model random variables $R$ and $D$ are independent, then $N$ is necessarily a Poisson random variable. Actually, it was shown in \cite{rao1964characterization}  that in the binomial damage model a weaker assumption of partial independence, known as the Rao-Rubin condition,
  \begin{equation}\label{RR}
\Pr_{R|D=0}=\Pr_R,
  \end{equation}
implies that $N$ is a Poisson random variable.

  In this context, formula \eqref{q-thin} for $q=\rho^2$ defines a damage model with function
   \begin{equation}\label{qthin}
  s(r,n)=\tfrac{q^r}{[n+1]_q},\quad r=0,1,\ldots,n\ge 0.
  \end{equation}

  Corollary \ref{C-BD-T2.8} can be interpreted as follows.
 If $N$ is a $q$-negative binomial as given in \eqref{dunkl}, then the damaged $D$ and undamaged $R$ parts of $N$ are independent and follow respectively $\mathrm{geo}(q\theta)$ and $\mathrm{geo}(\theta)$ distributions. It appears that under \eqref{qthin} the Rao-Rubin condition implies that $N$ is of the form given in \eqref{dunkl}.   More precisely, we have the following characterization result.
 \begin{proposition}
Let $R$ and $N$ be non-negative integer valued random variables such that $\Pr(R=r|N=n)=s(r,n)$ as in \eqref{qthin}, $q>0$. Assume that $\Pr(D=0)=1-\theta\in(0,1)$ and $q\theta\in(0,1)$.

If the Rao-Rubin condition \eqref{RR} holds, then $R\sim \mathrm{geo}(q\theta)$ and $D=N-R\sim\mathrm{geo}(\theta)$
are independent, that is, $N$ is $q$-negative binomial as given in \eqref{dunkl}.
  \end{proposition}
  \begin{proof}

  Condition  \eqref{RR} says
  \[
  \Pr(R=r,D=0)=\Pr(D=0)\Pr(R=r),\quad r=0,1,\ldots.
  \]
  Applying \eqref{qthin} to both sides above and canceling the common term $q^r$   we get
  \[
\tfrac{\Pr(N=r)}{[r+1]_q} =
(1-\theta) \sum_{n\geq r} \tfrac{\Pr(N=n)}{[n+1]_q}.
\]
Subtracting sidewise two such identities for $r$ and $r-1$, $r\geq 1$, we obtain recursion
  \[
  \tfrac{\Pr(N=r)}{[r+1]_q}=\theta \tfrac{\Pr(N=r-1)}{[r]_q},\quad r=1,2,\ldots,
  \]
  which gives
  \[
  \Pr(N=r)=[r+1]_q\,\theta^r\,\beta,\quad r=1,2,\ldots,
  \]
  where $\beta=\Pr(N=0)\neq0$.

  \end{proof}

Similar results for damage models that include
a result for $q=1$, i.e., for uniform damage and under some technical integrability conditions, can be found in \cite[Theorem 3.1]{patil1977characterizations}.

\subsection{A Poisson Example}
We give another explicit example of a pair satisfying \eqref{q-thin}.   If $X_0\sim 1+\mathrm{Poisson}(1)$ is a shift by one of a Poisson random variable  with parameter $\lambda=1$,
then $G$ from \eqref{q-thin} with $\rho=1$
is $\mathrm{Poisson}(1)$.
  Indeed,
  with $\Pr(X_0= n)=e^{-1}/( n-1)!$, $ n=1,2,\dots$, for $m\geq 1$ we get
  \begin{multline*}
     \Pr(G=m)=e^{-1}\sum_{ n=m}^\infty \tfrac{1}{( n-1)!( n+1)} =
  e^{-1}\sum_{ n=m}^\infty \tfrac{ n}{( n+1)!}
  \\
  =e^{-1}\sum_{ n=m}^\infty \pp{\tfrac{ n+1}{( n+1)!}-\tfrac{1}{( n+1)!}}
  =e^{-1}\sum_{ n=m}^\infty \tfrac{1}{ n!}-e^{-1}\sum_{ n=m+1}^\infty \tfrac{1}{ n!}=\tfrac{e^{-1}}{m!} .
  \end{multline*}
  Thus \eqref{Pit-rep} holds with symmetric random walk ($\rho=1$),  $G\sim \mathrm{Poisson}(1)$ and  $X_0\sim 1+\mathrm{Poisson}(1)$.

  Interestingly, $\wt G=X_0-G$ is %
  related to $X_0$ by the same formula, so it is $\mathrm{Poisson}(1)$, too. The pair $\wt G, G$ defines  bivariate distribution
  \[p(i,j)=e^{-1}\tfrac{i+j}{(i+j+1)!},   \quad i,j\in\ZZ_{\geq 0}\] with two $\mathrm{Poisson}(1)$   marginals, a $1+\mathrm{Poisson}(1)$  sum, and uniform laws of the marginals conditioned on the sum.

\appendix
 \section{Proof of Proposition \ref{Prop-T}}\label{sec:Prop-T}
The following identity is inspired by the proof of \cite[Lemma 3.1]{pitman1975one}.
\begin{lemma}\label{L2.4}
  For fixed $\vv x\in C\Zin{0,t}$ and $r\in\Zin{K_0(\vv x),x_t}$, let $\vv s\topp r(\vv x)=(s_0\topp r,\dots,s_t\topp r)$ be given by \eqref{sr}.
  Then
  \begin{equation}\label{min=max}
    \max_{i\in\Zin{0,j}}\left(s_i\topp r(\vv x)+K_0(\vv x)\right)_+ = r\wedge K_j(\vv x) -K_0(\vv x), \quad j\in\Zin{0,t}.
  \end{equation}
\end{lemma}
\begin{proof}
  Since $\vv x$   and $r\in\Zin{K_0(\vv x),x_t}$ are fixed in this proof, we   write   $\vv s\topp r(\vv x)=(s_0,\dots,s_t)$  and $K_0=K_0(\vv x)$,  suppressing the dependence on   parameters $\vv x$  and $r$.
   Denote %
  \begin{equation}\label{rj}
     r_j=r\wedge K_j(\vv x),\quad j\in\Zin{0,t},
  \end{equation}
  For $j=0$ both sides of \eqref{min=max} are $0$, as $r_0=K_0$. %

  To prove that \eqref{min=max} holds for all $j\in\Zin{0,t}$, we proceed by contradiction.
    Suppose \eqref{min=max} holds for some $j-1\in\Zin{0,t-1}$ but fails for $j$. Then writing $\wt s_j = s_j + K_0$, $j\in\Zin{0,t}$, we have
 \begin{equation*}
  \max_{i\in\Zin{0,j}}(\wt s_i)_+ =
       (\wt s_j)_+\vee\max_{i\in\Zin{0, j-1}}(\wt s_i)_+
       =
       \wt s_j\vee\max_{i\in\Zin{0, j-1}} \pp{\wt s_i}_+,
 \end{equation*}
 where in the last equality we used the fact that
 $\max_{i\in\Zin{0, j-1}}(\wt s_i)_+\geq 0$.
    By \eqref{min=max} used for $j-1$, the definition \eqref{sr} of $s_j$ and on noting that $r_{j-1}=r_j\wedge x_{j-1}$ we get

    \begin{multline*}%
      \max_{i\in\Zin{0, j}}\pp{\wt s_i}_+ = (2 r_j-x_j-K_0)\vee(r_{j-1}-K_0)
       =(2 r_j-x_j)\vee r_{j-1}-K_0  = (2r_j-x_j)\vee(r_{j}\wedge x_{j-1}) -K_0.
    \end{multline*}
Our assumption that \eqref{min=max}  fails for $j$ implies that $(r_{j}\wedge x_{j-1})\vee (2r_j-x_j) \ne r_j$. Since $r_j\leq x_j$, it follows that
\[(r_{j}\wedge x_{j-1})\vee (2r_j-x_j)<r_j.\]
Consequently, $2r_j-x_j<r_j$  and $x_{j-1}<r_j$. Thus
$x_{j-1}<r_j<x_j$, which for integers implies $x_j-x_{j-1}\geq 2$, a contradiction. Thus \eqref{min=max} holds for all $j\in\Zin{0,t}$.
\end{proof}
Formula \eqref{min=max} holds also under a different set of assumptions.
\begin{lemma} Let $\vv x\in C\Zin{0,t}$ and $K_0:=K_0(\vv x)=\min_{j\in\Zin{0,t}}x_j$.
   If $\vv x=\TT_{-K_0}(\vv s)$ and $r=R(\vv s, \vv x)$ is given by
   \begin{equation}\label{s2r}
r:=R(\vv s,\vv x):=\max_{i\in\Zin{0, t}}(s_i +K_0)_++K_0.
\end{equation}
  then $r\in\Zin{K_0,x_t}$ and
  \begin{equation}\label{min=max2}
  \max_{i\in\Zin{0, j}}(s_i+K_0)_+ =r\wedge K_j(\vv x)-K_0,\qquad j\in\Zin{0,t}.
  \end{equation}
\end{lemma}
\begin{proof}
To shorten the formulas below, we again use the notation $\wt s_j = s_j + K_0$, $j\in\Zin{0,t}$, introduced in the previous proof. (Note, however, that previously we used this notation with $\vv s\topp r$ and now we use it with $\vv s$.) We also use $r_j$ defined by  \eqref{rj}.

Clearly,  $r\geq K_0$. To confirm that $r\leq x_t$, we note that $\mathbf x=\mathbb T_{-K_0}(\mathbf s)$, in view of \eqref{s2r}, gives
\[x_t=2\max_{i\in\Zin{0, t}} (s_i+K_0 )_+-s_t=
\pp{\max_{i\in\Zin{0, t}} (s_i+K_0 )_+- (s_t+K_0) }+r\geq  r.\]

Next,we prove \eqref{min=max2}.  For $j=t$, \eqref{min=max2} holds since $r\wedge K_t(\vv x)=r\wedge x_t=r$, see again \eqref{s2r}.
To prove that \eqref{min=max2}  holds for all $j\in\Zin{0,t}$, we proceed by contradiction. Suppose \eqref{min=max2} holds for some $j\in\Zin{1,t}$ but not for $j-1$.
Then, see \eqref{rj},
\begin{multline*}
r_{j-1}-K_0=r_j\wedge x_{j-1}-K_0=(r_j-K_0)\wedge (x_{j-1}-K_0)=
\pp{\max_{i\in\Zin{0, j}}(\ \wt s_i )_+}\wedge \pp{x_{j-1}-K_0}
\\
=
\pp{\max_{i\in\Zin{0, j}}(\ \wt s_i )_+}\wedge \pp{2 \max_{i\in\Zin{0, j-1}}(\ \wt s_i )_+-\ \wt s_{j-1}} ,
\end{multline*}
where the penultimate equality follows by \eqref{min=max2} for $j$ and  the last equality follows from $\vv x=\TT_{-K_0}(\vv s)$, see \eqref{def-T}.
 Since \[\max_{i\in\Zin{0, j}}(\wt s_i)_+=\pp{\max_{i\in\Zin{0, j-1}}(\wt s_i)_+}\vee (\wt s_j)_+=\pp{\max_{i\in\Zin{0, j-1}}(\wt s_i)_+}\vee \wt s_j,\]
(in the latter equality we used the fact that $c:={{\max_{i\in\Zin{0, j-1}}(\wt s_i)_+}}\geq 0$),  we obtain
\begin{equation*}\label{W*}
r_{j-1}-K_0
=\pp{c\vee \wt s_j}
\wedge \pp{2 c-\wt s_{j-1}} .
\end{equation*}

  As  we assume that \eqref{min=max2} does not hold  for $j-1$, it follows that
$\pp{c\vee \wt s_j}
\wedge \pp{2 c-\wt s_{j-1}}\ne c$. Since $\wt s_{j-1}\leq c$, we see that
$\pp{c\vee \wt s_j}
\wedge \pp{2 c-\wt s_{j-1}}> c$. Consequently, $\wt s_j>c$ and $2 c-\wt s_{j-1}>c$, whence $\wt s_j>c >\wt s_{j-1}$. For integers, this implies $\wt s_j-\wt s_{j-1}=s_j-s_{j-1}\geq 2$, which is impossible for $\vv s\in C\Zin{0,t}$. Thus \eqref{min=max2}  holds for all $j\in\Zin{0,t}$.

\end{proof}
\color{black}

\begin{lemma}\label{2.5}
  For all $t\in\mathds Z_{\ge 0}$   if $\vv x\in C\Zin{0,t}$ is given by $\vv x=\TT_g(\vv s)$ for some $\vv s\in C\Zin{0,t}$ and some $g\in\ZZ_{\geq 0}$, then with $K_0:=K_0(\vv x)$  given by \eqref{K-def} and  $\vv s \topp r(\vv x)$ defined by \eqref{sr}
we have one of the following cases:
\begin{enumerate}[(i)]
\item $g>-K_0$ and  $\vv s=\vv s\topp{r}(\vv x)=-\vv x$ with $r=K_0$.
\item  $g=-K_0$ and  $\vv s=\vv s\topp r(\vv x)$ with $r=R(\vv s,\vv x)\in\Zin{K_0,x_t}$.
\end{enumerate}
\end{lemma}
\begin{proof} We first verify that   $g\geq-K_0$. To see this, we prove that
 that for all $t\in\ZZ_{\geq 0}$, if $\vv x=\TT_g(\vv s)$ with $\vv s\in C\Zin{0,t}$, then
 \begin{equation}
    \label{K-trop}
    -K_0(\vv x)=g \wedge \max_{j\in\Zin{0,t}}  s_j.
\end{equation}
The proof is by induction on $t$.
Since $x_0=0$ and $g\geq 0$, both sides  of \eqref{K-trop} are zero for $t=0$. Suppose \eqref{K-trop} holds for some $t$. %
Then  for $\vv s\in C\Zin{0,t+1}$  and $C\Zin{0,t+1}\ni \vv x=\TT_g(\vv s)$
\begin{multline}\label{-K0}
   -K_0(\vv x)=(-\min_{j\in\Zin{0,t}} x_j)\vee (-x_{t+1}) =(-\min_{j\in\Zin{0,t}} x_j)\vee  (s_{t+1}-2 (\max_{j\in\Zin{0,t+1}}  s_j-g)_+)
  \\ \stackrel{\eqref{K-trop}}{ = }(g \wedge \max_{j\in\Zin{0,t}}  s_j)\vee  (s_{t+1}-2( \max_{j\in\Zin{0,t+1}}  s_j-g)_+).
\end{multline}
We now consider two cases.
\begin{enumerate}
\item Suppose that $g< \max_{j\in\Zin{0,t+1}}s_j$. Then,  in view of the fact that $\vv s \in C\Zin{0,t+1}$ we have
$g\leq \max_{j\in\Zin{0,t}}s_j$, so $ g \wedge \max_{j\in\Zin{0,t}}  s_j=g$. Since
\[ s_{t+1}-2 (\max_{j\in\Zin{0,t+1}}s_j-g)_+ =(s_{t+1}-\max_{j\in\Zin{0,t+1}}s_j)+(g-\max_{j\in\Zin{0,t+1}}s_j)+g<g,\]
we get $-K_0(\vv x)= g$.%

\item Suppose that $g\geq  \max_{j\in\Zin{0,t+1}}s_j$.  Then  $(\max_{j\in\Zin{0,t+1}}s_j-g)_+=0$ and $g\geq \max_{j\in\Zin{0,t}}s_j$, so from the right-hand side of \eqref{-K0} we get
\[ -K_0(\vv x)
=(g \wedge \max_{j\in\Zin{0,t}}  s_j))\vee s_{t+1}=\max_{j\in\Zin{0,t}}  s_j\vee s_{t+1}=\max_{j\in\Zin{0,t+1}}  s_j.%
\]
\end{enumerate}

Thus in both cases we get $-K_0(\vv x)= g\wedge \max_{j\in\Zin{0,t+1}}s_j$ and \eqref{K-trop} holds for $t+1$.
\color{black}

To prove (i), we observe that if $g>-K_0$ then \eqref{K-trop} implies that $\max_{j\in\Zin{0,t}} s_j=-K_0<g$. Therefore, $\vv x=\TT_g(\vv s)=-\vv s$. In view of \eqref{sr}, $\vv s\topp {K_0}(\vv x)=-\vv x$, so $\vv s=\vv s\topp {K_0}(\vv x)$ as required in {\em (i)}.

 Next, we prove (ii). Suppose $g=-K_0$.     For $\vv s$   and $\vv x=\TT_{-K_0}(\vv s)$, let $r=R(\vv s,\vv x)\in\Zin{K_0(\vv x),x_t}$ be given by \eqref{s2r}. Equation $\vv x=\TT_{-K_0}(\vv s)$, which is
$x_j=2\max_{i\in\Zin{0, j}}( s_i+K_0)_+-s_j$, becomes
$$s_j=2\max_{i\in\Zin{0, j}}( s_i+K_0)_+-x_j\stackrel{\eqref{min=max2}}{=} 2 (r_j-K_0) -x_j\stackrel{\eqref{sr}}{ = } s_j\topp {r}(\vv x).$$
This shows that    for $r=R(\vv s,\vv x)$   we have $\vv s=\vv s\topp r(\vv x)$.

\end{proof}

\begin{proof}[Proof of Proposition \ref{Prop-T}] \sloppy
 Set $K_0=K_0(\vv x)$. We first note that formula \eqref{sr} gives $\vv s\topp {K_0}(\vv x)=-\vv x$. In this case we have
$\max_{i\in\Zin{0, j}} s_i\topp {K_0}\leq \max_{i\in\Zin{0, t}} s_i\topp {K_0} =-K_0$ so
$(\max_{i\in\Zin{0, j}} s_i\topp {K_0}-g )_+=0$ for all $g\geq -K_0$. Formula \eqref{def-T} gives
\[\TT_g(\vv s\topp {K_0})=-\vv s\topp {K_0}=\vv x \quad \mbox{ for all } g\ge -K_0.\]
 This shows that $\TT$ is onto and that  inverse image of $\mathbf x$ under $\TT$ includes the set $\ccbb{ (g,\vv s\topp {K_0}(\vv x)): g\in\llbracket-K_0,\infty)}$.

Suppose now that $r\in\Zin{K_0+1,x_t}$ .
Then by \eqref{min=max} of Lemma \ref{L2.4},  for $\vv s\topp r=\vv s\topp r(\vv x)$ given by \eqref{sr} and $r_j$ given by \eqref{rj},  we have
\[\TT_{-K_0}(\vv s\topp r)= \pp{2(\max_{i\in\Zin{0, j}}s_i\topp r+K_0)_+-s_j\topp r}_{j\in\Zin{0,t}}
=\pp{2(r_j-K_0)-s_j\topp r}_{j\in\Zin{0,t}}=\vv x.\]
This shows that the set $\ccbb{ (-K_0,\vv s\topp r(\vv x)):
r\in\{K_0+1,K_0+2,\ldots,  x_t\}}$ is in the inverse image of $\vv x$ under  $\TT$.

Note that for a given $\vv x$, all  $\{\vv s\topp r(\vv x): r=K_0,K_0+1,\dots,x_t\}$  are distinct. Indeed, from \eqref{min=max} we have
$(\max \vv s\topp{r}+K_0)_+=r-K_0$ and therefore \eqref{s2r} gives $R(\vv s\topp {r},\vv x)=r$. Thus $\{\vv s\topp r\}$ are distinct, as they map to different values in $\{K_0,K_0+1,\ldots,x_t\}$.

The above shows that \eqref{Sets} is a subset of
the inverse image of $\vv x$ under  $\TT$.
By Lemma \ref{2.5} the converse inclusion holds, so the set \eqref{Sets} is the inverse image  $\TT^{-1}(\{\vv x\})$.

Next, we  prove \eqref{UDH2UDH}. In this proof, $r$ is fixed, so we drop superscript $r$ in the notation
$\vv s\topp r=(s_0,\dots,s_t)$. Since
by \eqref{sr} we have
$s_j=2(r_j-K_0)-x_j$, we see that
\[s_{j+1}-s_j=2 ( r_{j+1}-r_j)-(x_{j+1}-x_j)
=2(r_{j+1}-r_{j+1}\wedge x_j)-(x_{j+1}-x_j).\]
We now consider two cases: $x_j\geq r_{j+1}$ and $x_j<r_{j+1}$. In the first case, we
get $r_{j+1}\wedge x_j=r_{j+1}$, so
$s_{j+1}-s_j=x_j-x_{j+1}$. In the second case, $x_{j+1}-x_j=1$ as $x_j<x_{j+1}$. Since
$x_{j+1}-1=x_j<r_{j+1}\leq x_{j+1}$, we have $r_{j+1}=x_{j+1}$. We get
\begin{equation}
  \label{cases}
  s_{j+1}-s_j=\begin{cases}
  x_j-x_{j+1} & x_j\geq r_{j+1}, \\
 x_{j+1}-x_j= 1 &  x_j< r_{j+1}.
  \end{cases}
\end{equation}

We now note that  $x_j=x_{j+1}$ if and only if  $s_{j+1}=s_j$, so  $H(\vv s\topp r)=H(\vv x)$.

Next, define $j_v$ as the largest index $j$ such that  $x_{j}=K_0+v-1$, $v=1,\dots,r-K_0$.
Note that for $\vv x\in C\Zin{0,t}$ such $j_v$ are well defined, as the values of $\vv x$ must cover
the entire range $\{K_0,K_0+1,\ldots,x_t\}$ and $r\in\Zin{K_0,x_t}$. From $x_t>r-1$, we see that $j_v<t$.
For each such $j_v$ we have $x_{j_v}<r_{j_v+1}$, so by \eqref{cases}, $x_{j_v+1}-x_{j_v}=1$.
(To explain why $x_{j_v}<r_{j_v+1}$ we observe that $x_{j_v}$ is below $x_{j_v+1},\dots,x_t$ and $x_{j_v}=K_0+v-1\leq r-1<r.$)

We note that $\{j\ \in\Zin{0,t} : x_j<r_{j+1}\}=\{j_1,\dots,j_{r-K_0}\}$.
Indeed, let $A=\{j\ \in\Zin{0,t} : x_j<r_{j+1}\}$  and $ B=\{j_1,\dots,j_{r-K_0}\}$.  We already noted that $j_v\in A$, that is  $B\subset A$.
To prove the converse inclusion, take $j\in A$. Since $x_j<r_{j+1}$ thus $x_j\leq r-1$ so $j<t$. Thus there exists $v\in\{1,\dots,r- K_0\}$  such that $x_j=K_0+v-1$.
Since $x_j<r_{j+1}\leq \min\{x_{j+1},\dots,x_t\}$, we see that $j=j_v$ since $x_j$ is the last element on the level $K_0+v-1$.

 Thus we identified  $(r-K_0)$ up-steps of $\vv s^r$ which coincide with the up-steps of $\vv  x$. In view of \eqref{cases}, the remaining up-steps of $\vv s^r$ are all the down-steps of $\vv x$, i.e. $U(\vv s\topp r)=r-K_0+D(\vv x)$. The last formula in \eqref{UDH2UDH} comes from $U+D+H=t$.
\end{proof}

\subsection*{Acknowledgments}
 We would like to thank Yizao Wang for several stimulating discussions which in particular led to Theorem \ref{T-YZ}.
WB's research was partially supported by Simons Foundation  Award Number: 703475, US.
 JW's research was partially supported by grant  IDUB no. 1820/366/201/2021,  Poland, and
 National Science Center Poland [project no.  2023/51/B/ST1/01535].
The authors acknowledge support from the Taft Research Center at the University of Cincinnati.

\def\polhk#1{\setbox0=\hbox{#1}{\ooalign{\hidewidth
  \lower1.5ex\hbox{`}\hidewidth\crcr\unhbox0}}}

\end{document}